\documentclass{siamltex704}
\title{Steady water waves in the presence of wind} \author{Samuel Walsh\footnotemark[3]\ \footnotemark[2]
  \and Oliver B\"uhler\footnotemark[4]\ \footnotemark[2] \and
  Jalal Shatah\footnotemark[5]\ \footnotemark[2]}

\usepackage{amssymb,stmaryrd,amsmath}
\usepackage{amscd}
 \usepackage{mathrsfs}
\usepackage[dvips]{graphicx}

\newcommand{\be}{\begin{equation} }
\newcommand{\ee}{\end{equation}}
\newcommand{\bse}{\begin{subequations}}
\newcommand{\ese}{\end{subequations}}

\newcommand{\jump}[1]{\left\llbracket{#1}\right\rrbracket}
\newcommand{\pfthm}[1]{\emph{Proof of Theorem #1. }}

\newtheorem{remark}{Remark}

\usepackage{hyperref}
\hypersetup{
colorlinks=true, 
linkbordercolor={1 1 1}, 
citebordercolor={1 1 1} 
} 

\begin{document}
\maketitle

\renewcommand{\thefootnote}{\fnsymbol{footnote}}
\footnotetext[2]{Courant Institute of Mathematical Sciences, New York University, New York, N.Y. 10012}
\footnotetext[3]{walsh@\allowbreak cims.\allowbreak nyu.\allowbreak edu}
\footnotetext[4]{obuhler@\allowbreak cims.\allowbreak nyu.\allowbreak edu}
\footnotetext[5]{shatah@\allowbreak cims.\allowbreak nyu.\allowbreak edu}
\renewcommand{\thefootnote}{\arabic{footnote}}

\begin{abstract} In this paper we develop an existence theory for
  small amplitude, steady, two-dimensional 
  water waves in the presence of wind in the air above.  The presence of the wind is modeled by a Kelvin--Helmholtz type
  discontinuity across the air--water interface, and a
  corresponding jump in the circulation of the fluids there.
  We consider both fluids to be inviscid, with the water region
  being irrotational and of finite depth.  The air region is
  considered with constant vorticity in the case of infinite
  depth and with a general vorticity profile in the case of a
  finite, lidded atmosphere.
\end{abstract}

\begin{keywords} wind wave, water waves, traveling waves, bifurcation theory
\end{keywords}

\begin{AMS}\end{AMS}

\section{Introduction} \label{introductionsection}

Understanding the precise means by which the wind creates surface
waves in the ocean is both a central problem in geophysical fluid
dynamics and a famously difficult one.  From common experience,
of course, the fact that wind blowing over quiescent water will
lead to the generation of waves is fairly obvious.  Indeed, the
basic physical process appears straightforward: the jump in
tangential velocity across the air-water interface causes
instability, creating growing modes, which eventually stabilize
to become traveling waves.  Were this the case, we might hope
that the salient features of the system might be captured by the
Kelvin--Helmholtz (K--H) model (cf., e.g.,
\cite{drazin2004book}).  Intriguingly, this appears to be untrue.
Including the effects of surface tension, the K--H model predicts
that the speed of the wind must be above 660 cm/sec in order to
excite waves, which is an order of magnitude larger than what
observation suggests (cf. \cite{lamb1993hydrodynamics}).  The
discrepancy indicates that there are crucial components of the
system that has been overlooked by the K--H viewpoint.

The search for these missing features has led 
to a number of
competing models for the time-dependent wind-driven generation of
water waves with 
perhaps the most influential being that of Miles
(cf. \cite{miles1957windwaves1, miles1959windwaves2}), 
which is based on the existence of a \emph{critical layer} in the
continuous wind shear profile above the surface wave (we will
discuss some aspects of this model in
\S\ref{sec:stagn-crit-layers} below).

The authors will address the time-dependent problem in a later
work, but in the present paper we begin by consider the steady
problem.  That is, we investigate the question of existence of
traveling waves in a two phase air--water system.  We endeavor to
do this in such a way that we can view the waves as the eventual
byproduct of wind blowing over water, though we shall remain
agnostic as to how exactly that generation took place.  For us,
this means that the circulation as a function of the streamlines
should be discontinuous over the air-water interface.  Since the
circulation is conserved by the flow, this is a necessary
condition for the traveling wave to be dynamically accessible
from an initial configuration with laminar flow in the air and
water with a jump in the tangential velocity over the boundary.

In this section, we shall first describe the basic framework,
then make an informal statement of the results.  The precise
theorems will be given later in the text.

\subsection{Eulerian formulation} \label{eulerianformulationsection}

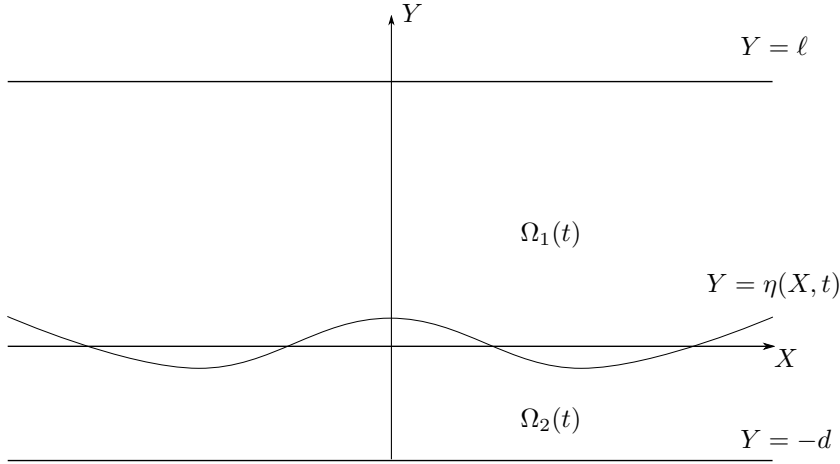
\begin{figure}[h] \setlength{\unitlength}{4.5cm} 
\centering
\includegraphics[scale=0.65]{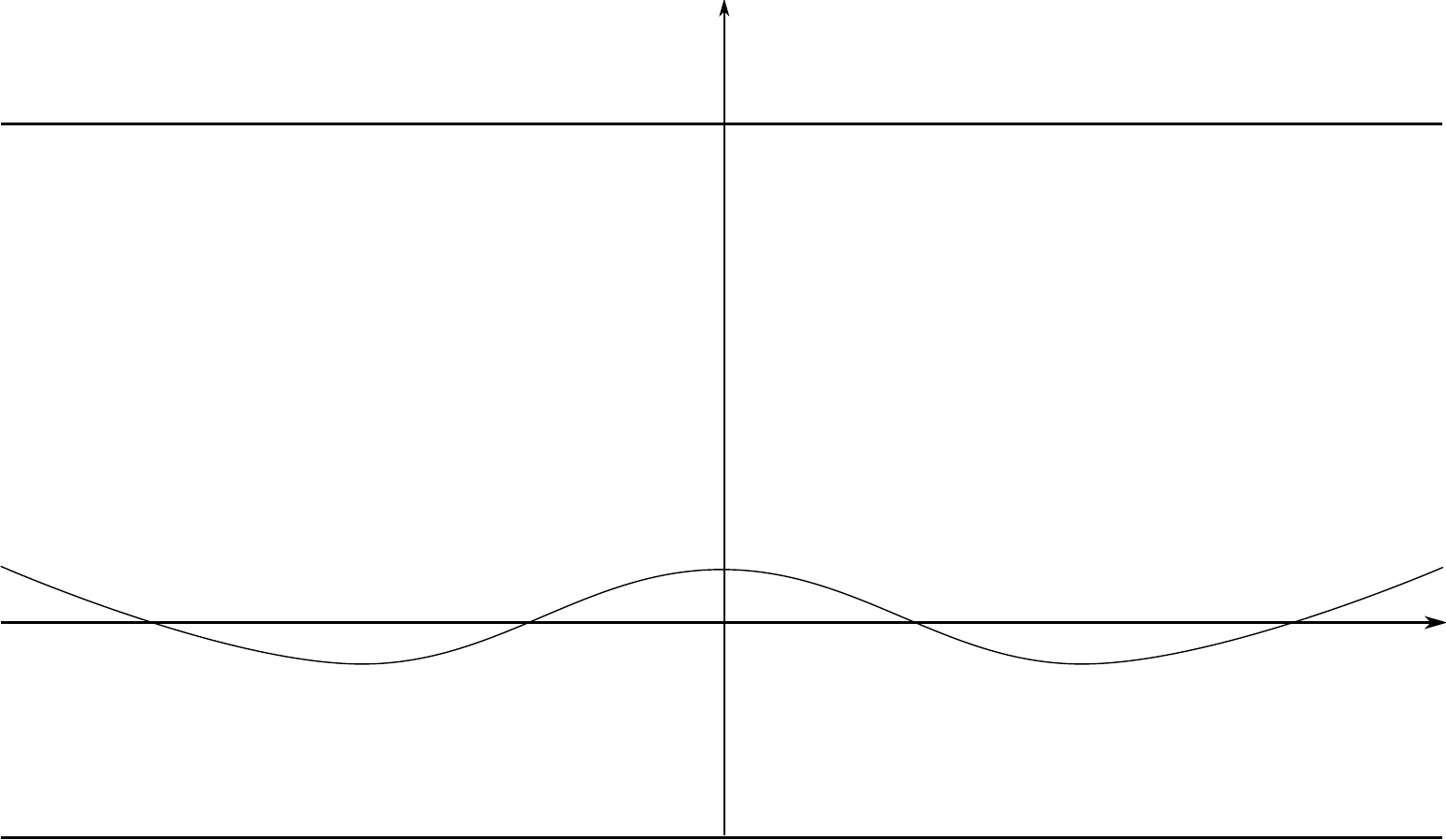}
\put(-0.2,0.50){$Y = \eta(X,t)$}
\put(-0.10,0.04){$Y = -d$}
\put(-0.10,1.20){$Y = \ell$}
\put(-0.00,0.28){$X$}
\put(-1.1,1.30){$Y$}
\put(-0.75,0.10){$\Omega_2(t)$}
\put(-0.75,0.65){$\Omega_1(t)$}
\caption{The fluid domain $\Omega(t)$.  The wave is assumed to propagate to the left with speed $c$.}
\end{figure}

Consider a two-dimensional surface wave in the ocean, with an
accompanying wave in the atmosphere above.  Fix a Cartesian
coordinate system $(X,Y)$ so that the $X$-axis points in the
direction of wave propagation, and the $Y$-axis is vertical.  We
assume that the bed of the ocean is flat and occurs at $Y = -d$,
while the interface between the wave and the atmosphere is a free
surface given by the graph of a function $\eta = \eta(X,t)$.  We
are particularly interested in the case where $\eta$ is periodic
in $X$, for fixed $t$.  While these are not physical (they have
infinite energy), they are the type typically studied in
connection with linear stability.  Not coincidentally, they are
also more amenable to mathematical analysis since there is a gain
of compactness in the $X$-direction.  We may normalize $\eta$ by
choosing the axes so that the free surface is oscillating around
the line $Y = 0$.  The atmospheric domain can be thought of as
either unbounded or bounded in $Y$.  The unbounded regime models
the situation in which the characteristic horizontal length scale
is vastly smaller than the vertical length scale.  On the other
hand, if the dynamics of the wave away from the interface are
deemed of secondary importance, a common practice is to take the
air region to be bounded above by a rigid lid at $Y = \ell$.  All
told, the fluid domain at a given time $t$ is
\[ \Omega(t) = \Omega_1(t) \cup \Omega_2(t),\]
where
\[ \Omega_1(t) :=  \{ (X,Y) \in \mathbb{R}^2 : \eta(X,t) < Y < \ell \}\]
is the air region ($\ell = +\infty$ in the unbounded case), and
\[\Omega_2(t) := \{ (X,Y) \in \mathbb{R}^2 : - d < Y < \eta(X,t) \} \]
is the water region.  The total width of the channel is thus $W
:= \ell + d$.  In what follows, we shall assume that $\ell$ is
fixed at the outset while $d$ is 
to be determined during the solution procedure.  We point out
also that we are not including the air--sea interface $\{ Y =
\eta(X,t) \}$ in the fluid domain, and thus $\Omega(t)$ is not
connected.

Let $u = u(X,Y,t)$ and $v = v(X,Y,t)$ denote the horizontal and
vertical fluid velocities, respectively.  Let $\varrho =
\varrho(X,Y,t) > 0$ be the density.  We assume that velocity
field is incompressible \be u_X + v_Y = 0, \qquad \textrm{in }
\Omega(t). \label{incompress} \ee Conservation of mass is
enforced by stipulating that the density of each fluid particle
is likewise preserved.  For an inviscid fluid, this is equivalent
to stating that the material derivative of $\varrho$ is zero: \be
\varrho_t + u \varrho_X + v \varrho_Y = 0, \qquad \textrm{in }
\Omega(t). \label{consmass1} \ee In this paper, we exclusively
consider the case where the density in the air and water regions
is constant, that is,
\[ \varrho(t,X,Y) = \rho_{\textrm{air}}
\chi_{\Omega^{(1)}(t)}(X,Y) + \rho_{\textrm{water}}
\chi_{\Omega^{(2)}(t)}(X,Y),\] where $\chi_{\Omega^{(i)}(t)}$ is
the characteristic function for the fluid region
$\Omega^{(i)}(t)$, $i = 1,2$, and $\rho_{\textrm{air}}$,
$\rho_{\textrm{water}}$ are the given densities of water and air,
respectively.  Thus \eqref{consmass1} will always be satisfied.
Note that because $\Omega(t)$ does not include the air--sea
interface, the above equation will hold in the classical
sense.  

The momentum equations for non-diffusive heterogeneous fluids are
the Euler equations, \be \left \{ \begin{array}{lll}
    \varrho u_t + \varrho u u_X + \varrho v u_Y  & = & -{P_X}{}, \\
    \varrho v_t + \varrho u v_X + \varrho v v_Y & = & -{P_Y}-
    g\varrho, \end{array} \right.  \qquad \textrm{in }
\Omega(t), \label{euler1} \ee where $P=P(X,Y,t)$ is the pressure
and $g$ is the gravitational constant.  Again, given our choice
of $\varrho$, one may alternatively view this as being satisfied
in $\Omega^{(1)}(t)$ and $\Omega^{(2)}(t)$ separately, with
$\varrho$ taking the appropriate constant value.

Let $\mathcal{I}(t) := \partial\Omega_1(t)
\cap \partial\Omega_2(t)$ denote the interface between the air
and water regions.  For simplicity, we shall use the convention
that the restriction of any quantity defined on $\Omega(t)$ or
$\Omega(t) \setminus \mathcal{I}(t)$ to the subset $\Omega_i(t)$
is denoted by a superscript $(i)$.  Thus, for example $u^{(i)} :=
u|_{\Omega_i(t)}.$ Similarly, we define the jump operator
\[ \jump{\cdot} := (\cdot)^{(1)}_{\mathcal{I}(t)} -
(\cdot)^{(2)}_{\mathcal{I}(t)}.\]

The dynamic boundary condition states that, ignoring the effects
of surface tension, the pressure must be continuous across the
interface.  Stated in terms of the operator $\jump{\cdot}$, this
is simply the statement that \be \jump{P} = 0, \qquad \textrm{on
} Y = \eta(X,t). \label{presscond} \ee

To couple the evolution of the boundary to that of the flow, we
impose a kinematic condition.  More precisely, we suppose that
$\mathcal{I}(t)$ is a material interface. This is enforced by
requiring that the normal velocity of the boundary agrees with
the normal velocity of the fluid.  Since we are assuming a graph
geometry for the free surface, we can express this quite
explicitly: \be v = \eta_t + u \eta_X, \qquad \textrm{on } Y =
\eta(X,t). \label{kincond1} \ee Similarly, in the lidded regime,
both upper and lower boundaries are unmoving, we must have \be v
= 0, \qquad \textrm{on } Y = -d. \label{bedcond} \ee \be v = 0,
\qquad \textrm{on } Y = \ell. \label{lidcond} \ee When $\Omega_2$
is unbounded, the bed condition \eqref{bedcond} remains valid,
but \eqref{lidcond} must be replaced with \be (u,v) \to
(u_\infty, 0) \textrm{ as } Y \to \infty, \textrm{ uniformly in
  $X$ and $t$, for some $u_\infty \in
  \mathbb{R}$.} \label{unliddedcond}\ee

Traveling wave solutions to \eqref{incompress}-\eqref{bedcond}
are those where, for some wave speed $c > 0$, the change of
variables
\[ x = X - ct, \qquad y = Y, \] eliminates time dependence.
Physically, this means that viewed from a frame of reference
moving with fixed speed $c$ in the direction of propagation,
$(u,v,\varrho, \eta, P)$ all appear steady.  Reusing notation,
from here on we shall simply identify $(u,v,\varrho, \eta,P)$
with their stationary profiles.  Observe that periodicity of the
traveling wave is equivalent to periodicity of the steady
quantities in the $x$-variable.  We shall therefore require that
$(u,v,\varrho, \eta,P)$ are $L$-periodic in $x$, for some $L >
0$.

In the moving frame \eqref{incompress}-\eqref{euler1} become \be
\left \{ \begin{array}{lll}
    u_x + v_y & = & 0, \\
    \varrho (u-c)u_x + \varrho v u_y  & = & -P_x ,\\
    \varrho(u-c) v_x + \varrho v v_y & = & -P_y - g
    \varrho, \end{array} \right. \qquad \textrm{in }
\Omega, \label{euler2} \ee where $\Omega$ is the steady domain.
The kinematic and dynamic conditions for the lidded problem are
likewise \be \left \{ \begin{array}{llll}
    v & = & 0, & \textrm{on } y = \ell, \\
    v & = & 0, & \textrm{on } y = -d, \\
    v & = & (u-c) \eta_x, & \textrm{on } y = \eta(x), \\
    \jump{P} & = & 0, & \textrm{on } y =
    \eta(x). \end{array} \right. \label{boundcond} \ee The
unbounded atmosphere case differs only in the condition at $y =
\infty$, where we require \be (u,v) \to (u_\infty, 0) \textrm{ as
} y \to \infty, \textrm{uniformly in $x$ for some $u_\infty \in
  \mathbb{R}$}. \label{steadyunliddedcond}\ee Recall, also, that
we have chosen our axes so that $\eta$ oscillates around the line
$y = 0$: \be \frac{1}{L}\int_{-L/2}^{L/2} \eta(x) \, dx =
0. \label{normalsurface} \ee
%

We note in passing that the steady Euler equations have an important consequence for the
Eulerian-mean momentum flux defined as
\begin{equation}
  \label{eq:1}
  F_E(y)  = \frac{1}{L}\int\limits_{-L/2}^{L/2} \rho (u-c)v \, dx.
\end{equation}
Physically, the function $F_E(y)$ is the mean upward flux of
horizontal momentum across a line of constant altitude $y$.  The
steady Euler equations imply that $dF_E/dy=0$ and hence $F_E$
does \textsl{not} depend on altitude $y$.  This
constant-momentum-flux result is also obvious from a physical
point of view: any vertical variation of $F_E$ would imply a
time-dependent accumulation of horizontal momentum within some
altitude range, which would be incompatible with the assumption
of a steady flow field.  Moreover, the boundary condition $v=0$
at the rigid lower boundary actually implies that $F_E=0$ there
and therefore $F_E=0$ throughout the domain.  This condition on
$F_E$ must be satisfied by all solutions to the steady equations.

Finally, another physical quantity of interest in the context of
wind-driven waves is the mean horizontal drag force exerted by
the air on the water across the interface $y=\eta(x)$.  This drag
force equals minus the Lagrangian-mean momentum flux\footnote{Eulerian and Lagrangian momentum
flux definitions and their role in wave dynamics are discussed in
detail in \cite{buhler09}} across the
undulating air--water interface, which is
\begin{equation}
  \label{eq:2}
  \mbox{drag force}\quad= -F_L  =
  \frac{1}{L}\int\limits_{-L/2}^{L/2} \eta_x(x) P(x,\eta(x)) \,
  dx. 
\end{equation}
It is easy to show by integrating the steady Euler equations over a
control volume with lower boundary $y=\eta(x)$ and an upper
boundary of any constant altitude $y$ above the interface that
\begin{equation}
  \label{eq:3}
  F_E = F_L = 0
\end{equation}
holds for steady waves, i.e., the Eulerian momentum flux equals
minus the drag force across the interface, and both are zero for
steady waves.   Conversely, a 
nonzero drag force is incompatible with a steady flow, and this observation lies at the
heart of Miles's theory for wind-driven water waves, which is
briefly discussed in the next section.

\subsection{Critical layers and Miles's theory}
\label{sec:stagn-crit-layers}

From \eqref{euler2} and \eqref{boundcond} it is clear that points
where $u = c$ are of special importance.  When this occurs, the
relative horizontal 
velocity appears to vanish in the moving frame, meaning that the
horizontal velocity of the fluid particle at that point matches
that of the wave.  This scenario we refer to as
\emph{stagnation}.  Note that this differs from standard usage,
since for us only the relative horizontal velocity needs to be
zero.  In the classical theory of steady water waves, stagnation
is closely associated with a loss of regularity stemming from the
degeneracy of the governing equations.  The most well-known
instance of this phenomenon is the so-called \emph{extreme waves}
of Stokes (cf., e.g., \cite{amick1982stokes}).  Stagnation points
play an even more central role in the time-dependent theory of
wind-driven water 
waves.  Note that when a flow is laminar, i.e. it is 
of the form $(u,v) = (U(y), 0)$, the critical points will arrange
themselves in horizontal lines $y=y_c$, say, 
such that $U(y_c)=c$.  These lines are called \emph{critical layers}.
The central thesis 
of Miles's theory is that the presence of a critical layer
in the air region provides a mechanism for the wind to impart horizontal 
momentum on the water via a nonzero drag force in \eqref{eq:2},
and it is precisely this drag force that is responsible for
creating linear instability at slower wind speeds than predicted
by K--H.  (cf. \cite{miles1957windwaves1}). 

Specifically, if the vertical momentum flux $F_E$ at the upper
domain boundary is zero, which is consistent with a lidded
domain, and if there is only a single critical layer, then linear
wave theory for weakly unstable waves predicts that $F_E$ has a
jump discontinuity across $y=y_c$ such that the drag force
on the
water is proportional to the ratio
$(-U^{\prime\prime}/U^{\prime})$ evaluated at the critical level
$y=y_c$ (cf. \cite{miles1957windwaves1,drazin2004book}).
This makes obvious the crucial role played by a nonzero value of
the second derivative $U''(y_c)$ at the critical layer, without
which there could be no wave growth induced by the critical layer.   Indeed, if
$U^{\prime\prime}(y_c)$ is assumed to vanish, the critical layer
is neutered or inactive: the momentum flux $F_E$ is then
continuous across $y=y_c$ and the
critical layer plays no part in the generation of
waves.  

Now, for the purpose of studying steady waves one must either
rule out critical layers \emph{a priori}, or one must allow only
inactive critical layers by requiring that $U''(y_c)=0$.  Indeed,
the existence theory we present in \S\ref{unboundedshearsection}
for the constant vorticity unbounded atmosphere case allows for
these inactive critical layers.  
  
The majority of the present paper is concerned with the lidded
atmosphere case, and in that we regime we restrict our attention
to waves \emph{without} stagnation and hence without critical
layers. 
The alternative approach of studying the existence of small
amplitude traveling water waves with stagnation points in various
regimes has recently been considered in several works
(cf. \cite{wahlen2009critical,ehrnstrom2010interior,ehrnstrom2010multiple,lin2011inviscid}).
In each of these papers, the authors prove that there exists a
(local) 
curve of non-laminar flows bifurcating from a background laminar
flow with a critical layer.  Though it is not always stated
explicitly, 
in every case this is done under the assumption that the critical
layer occurs 
only at an inflection point of $U$ 
and therefore $U^{\prime\prime}(y_c) = 0$.  If a similar
restriction is taken, we believe it would be possible to
generalize the results (in lidded domains) 
given here to allow for critical layers in the background flow.
As the above discussion makes clear, however, in the context of wind-driven waves these types of profiles
are not considerably more interesting than those without
stagnation. 

\subsection{Stream function and circulation}
\label{sec:stre-funct-circ}

Assuming an absence of stagnation, \eqref{euler2} ensures that we
may define a function $\psi = \psi(x,y)$ by \be \psi_x =
-\sqrt{\varrho}v,\qquad \psi_y = \sqrt{\varrho} (u-c), \qquad
\textrm{in } \Omega. \label{defpsi} \ee This $\psi$ is known as
the (relative) \emph{pseudo-stream function} for the flow.  Here
we have the addition of a factor of $\varrho$ to the typical
definition of the stream function for an incompressible fluid.
This is meant to account for the inertial effects of the density
variation
(cf. \cite{walsh2009stratified,yih1965dynamics}). Observe also
that the restriction $u < c$ throughout the fluid becomes the
requirement: \be \psi_y < 0. \label{nostagnationcondition} \ee It
is evident from \eqref{defpsi} that $\psi$ is indeed a (relative)
stream function in the sense that $\nabla^\perp \psi$ is
collinear with the vector field $(u-c,v)$.  In other words, the
level sets of $\psi$, which we call \emph{streamlines}, coincide
with the standard definition of streamlines in the literature.

The kinematic condition in \eqref{boundcond} implies precisely
that the free surface, bed, and lid are each level sets of
$\psi$.  As \eqref{defpsi} only determines $\psi$ up to a
constant, we may take $\psi = 0$ on the upper lid, so that $\psi
= -p_0$ on $y = -d$, where $p_0$ is the \emph{(relative)
  pseudo-volumetric mass flux}: \be p_0 := \int_{-d}^{\ell}
\sqrt{\varrho(x,y)}\left[u(x,y)-c \right] \, dy. \label{defp0}
\ee It is straightforward to show that $p_0$ is a (strictly
negative) constant, i.e., it does not depend on $x$
(cf. \cite{walsh2009stratified}).  Physically, $p_0$ describes
the rate of fluid moving through any vertical line in the fluid
domain (and with respect to the transformed vector field
$\sqrt{\varrho}(u-c,v)$.)  Notice that, although we shall allow
$u$, $v$, and $\rho$ to have discontinuities across
$\mathcal{I}$, we assume that the pseudostream function is of class
$C(\overline{\Omega})$.  This is not an assumption:  because $\psi$ is defined by \eqref{defpsi} only up to a constant in each $\Omega^{(i)}$, we may without loss of generality take it to be continuous across the interface.

The conservation of mass in \eqref{euler2} implies that $\nabla
\varrho$ is orthogonal to the velocity field, and hence we may
let $\rho:[-p_0,0] \to \mathbb{R}^{+}$ be given such that \be
\varrho(x,y) = \rho(-\psi(x,y)) \label{defrho} \ee throughout the
fluid. We shall refer to $\rho$ as the \emph{streamline density
  function}, though one may alternatively view it as the
Lagrangian density.  In the ocean, one typically has that density
is increasing with depth, meaning that the lighter fluid elements
rest on the heavier ones.  Indeed, several physical mechanisms
work independently to enforce this situation: gravity naturally
leads to increased salinity near the bed, while temperature
decreases substantially as one moves deeper into the ocean where
the effects of the sun's heating are attenuated.  However, even
near the surface, water is on the order of 1000 times as dense as
the air.  The variations in the density are within the air and
water region are nowhere near as great as that across the
interface, and so we shall suppose that $\varrho$ is constant in
each region,
\[ \varrho^{(1)} \equiv \rho_{\textrm{air}}, \qquad \varrho^{(2)}
\equiv \rho_{\textrm{water}}.\] Stable stratification in this
case simply means $\rho_{\textrm{air}} < \rho_{\textrm{water}}$.

Conservation of energy can be expressed via Bernoulli's theorem,
which states that the quantity
\[ E := P + \frac{\varrho}{2}\left( (u-c)^2 + v^2\right) + g\rho
y, \] is constant along streamlines (see \cite{walsh2009stratified} for an elementary proof.)  If we
evaluate the jump of 
$E$ on the interface, we may use the dynamic
boundary condition to express the pressure in terms of $(u, v,
\eta)$, which gives rise to the following \be \jump{|\nabla
  \psi|^2} + 2g \jump{\rho} \left(\eta+d\right) = Q,
\qquad \textrm{on } y = \eta(x) \label{capgravdefQ} \ee where the
constant $Q := 2\jump{E+ g\rho d}$ gives roughly the jump
in the energy density across the free surface of the fluid.  We treat $Q$ as our parameter of bifurcation.

By taking the curl of the steady Euler's equations, one arrives at the identity
\[ \{ -\Delta \psi, \, \omega \} = 0 \qquad \textrm{in } \Omega, \]
where $\{\cdot, \cdot\}$ is the Poisson bracket, and $\omega$ is the scalar vorticity $\omega := v_x - u_y$.
Under the assumption that $u < c$ throughout the fluid, this allows us to conclude that there exists a single-valued function $\gamma$, called the \emph{vorticity strength function}, such that
\[  \omega(x,y) = \gamma(\psi(x,y)) \qquad \textrm{in } \Omega.\]

The final ingredient in our model is a condition on the
circulation on each streamline, namely that the circulation in
the air and water regions do not coincide.  This is meant to
ensure that the waves we construct are dynamically accessible
from an initial configuration of a shear profile wind blowing
over water.  Since the circulation will be constant along a
streamline, we define the function $\Gamma = \Gamma(p)$ defined
by \be \Gamma(p) := \int_{\mathscr{S}(p)} (u(x,y), v(x,y)) \cdot
d\mathbf{x} = \Gamma(p), \qquad \textrm{for } p_0 < p <
0, \label{euleriancirculationcond} \ee here $\mathscr{S}(p) := \{
(x,y) \in \Omega : -L/2 < x < L/2, \, \psi(x,y) = -p\}$.  For us,
then, the presence of wind means that there is a discontinuity in
$\Gamma$ at the streamline corresponding to the air-water
interface.

\subsection{Informal statement of results}
Now that the basic objects have been introduced, we can give a brief summary of our results.  The precise theorem statements are presented later.  

\begin{romannum}
\item Consider the existence of steady wind-driven water waves
  with a lidded atmosphere and irrotational flow in both the air
  and water and without stagnation.  Fix the period $L$, density
  jump $\jump{\rho}$, (pseudo) volumetric mass flux in the water
  region $p_1$, the (pseudo) relative circulation on the lid
  $\Gamma_{\mathrm{rel}}$, and the height of the lid $\ell$.
  Then, if a certain compatibility condition is met
  \eqref{compatibilitycond}, there exists a family of laminar
  flows that are exact solutions.  Moreover, if a size condition
  is satisfied \eqref{ideallbc}, there is a curve of small
  amplitude (classical) exact solutions bifurcating from this
  family.  See \S\ref{liddedidealsection} and Theorem
  \ref{ideallocalbifurcation}.

\item Consider the situation as in (i), but where the flow in the air region is rotational.   There is a corresponding compatibility condition relating $\gamma$, $\Gamma_{\textrm{rel}}$, $p_1$ and $\ell$ \eqref{shearcompatibilitycond}.  If it is satisfied, there exists of a one-parameter family of laminar flows, each with (pseudo) relative circulation $\Gamma_{\mathrm{rel}}$.  Moreover, under a certain local bifurcation condition \eqref{lbcnu} (or size condition \eqref{shearsizecond}), there is a curve of small amplitude (classical) exact solutions bifurcating from this family. See \S\ref{liddedvorticalsection} and Theorem \ref{shearlocalbifurcationtheorem}.

\item In the unbounded atmosphere regime, we fix the depth of the ocean $d$ and consider the existence of waves where the water region is irrotational and the wind has constant vorticity $\gamma_0$.   There is a family of laminar flows, parameterized (essentially) by the circulation at $y = +\infty$, and, under analogous bifurcation conditions (\eqref{unboundedideallbc} for $\gamma_0 = 0$ and \eqref{unboundedshearlbc}, otherwise), there are curves of small-amplitude (classical) exact solutions bifurcating from this family.  See \S\ref{unboundedidealsection}--\ref{unboundedshearsection} and Theorem \ref{unboundedideallocalbifurcationtheorem} and Theorem \ref{unboundedshearlocalbifurcationtheorem}.    Note that these results do not require an absence of stagnation.  In fact, when $\gamma_0 < 0$, there will be a critical layer in each of the background laminar flows.  \\ \end{romannum}

\noindent Notice that in each of (i) and (ii) there are hypotheses relating $\Gamma_{\textrm{rel}}$, $\ell$, $\gamma$, and $p_1$.  This is to be expected, in fact they can be viewed as a consequence of Stokes' theorem.  In our work, we elect to fix $p_0$ and $p_1$, as well as $\ell$ in the lidded case.  When the atmosphere is irrotational, these choices determine $\Gamma_{\textrm{rel}}$ by \eqref{compatibilitycond}; if the atmosphere is rotational, we take $\gamma$ to be fixed, and define $\Gamma_{\textrm{rel}}$ according to \eqref{shearcompatibilitycond}.   These choices are of course arbitrary and one can instead choose $\Gamma_{\textrm{rel}}$ and use \eqref{shearcompatibilitycond} to define $\gamma$ and $\ell$.  

Let us now briefly discuss the place of these results in the
existing literature.  The bifurcation theory techniques that we
employ have a long history in the study of steady water waves.
For the lidded regimes (points (i) and (ii)), we consider a
reformulation of the problem in semi-Lagrangian coordinates,
which has been used in a number of works, notably Amick--Turner
(cf., e.g.,
\cite{turner1981internal,turner1984variational,amick1984semilinear})
and Constantin--Strauss (cf. \cite{constantin2004exact}). Our
approach follows the latter in relying on elliptic estimates
rather than variational techniques.  More precisely, the method
we employ can be viewed as an adaptation of that in
\cite{walsh2009stratified} to the case of a non-continuously
stratified fluid in a channel, and with additional considerations
involving the (pseudo) relative circulation.  A similar problem
was considered by Amick and Turner in \cite{amick1986global},
ignoring the important issue of the circulation.  They, however,
were primarily interested in the solitary wave case, and so
developed the periodic existence theory only in order to obtain
solitary waves as a limit as the period goes to $+\infty$.  As a
consequence, this requires them to make certain assumptions on
$\gamma$ (otherwise the limiting wave will not be irrotational
and quiescent at $x = \pm\infty$); we do not impose any such
restriction.  It should be pointed out that the results of
Amick--Turner, as well as those of Constantin--Strauss, are
global, in the sense that the bifurcating curve of solutions is
extended to include waves with finite amplitude.  We believe that
similar result is possible in this case, since the basic
ingredients (mainly good Schauder-type \emph{a priori} estimates
for the elliptic equations involved) are available.  This is an
interesting and important question, but beyond the scope of the
preliminary investigations here.

Previous work on the infinite atmosphere case is comparatively
sparse.  In the applied literature, this is simply because the
lidded regime is seen as an adequate idealization of the infinite
atmosphere: so long as the 
wind curvature is evanescent and the air density is constant 
there is no mechanism for propagation of waves to or from
vertical infinity, and hence the dynamics far away from the
air--sea interface 
decay exponentially with altitude and 
are not thought to be particularly relevant.  Mathematically, of
course, removing the lid results in a loss of compactness, which
introduces some potentially serious difficulties.  Nonetheless,
we include as a simple application of our machinery a
mathematical treatment of the infinite atmosphere regime in the
case where the vorticity is constant in the air region.

\section{Formulation} \label{formulationsection} 

\subsection{Stream function formulation} \label{streamfunctionformulationsection}
The relevance of the pseudo-stream function and the vorticity strength function to the existence theory stems from the identity
\[ -\Delta \psi = \omega,\]
which, recalling the definition of the vorticity strength function leads to the semi-linear elliptic equation
\be -\Delta \psi = \gamma(\psi), \qquad \textrm{in } \Omega. \label{yiheq} \ee
One important consequence of the lack of stagnation, or, more accurately, the absence of \emph{active} critical layers, is that the Euler system can be recast as the above scalar problem.

Next, we note that the (steady) kinematic condition in \eqref{euler2} guarantees that the interface $\mathcal{I}$ is a streamline.  That is,
\[ \mathcal{I} = \{ \psi = -p_1\},\]
for some $p_1 > p_0$.  
The difference between $p_1$ and the value of $\psi$ at the top boundary of the domain gives the (pseudo) volumetric mass flux in the air region.  In the lidded case we thus take $p_1$ to be some fixed value and let $\psi|_{y = \ell} = 0$.  When the atmospheric region is unbounded, however, we let $\psi|_{\mathcal{I}} = 0$, and thus $\psi \to +\infty$ as $y \to +\infty$.  

Taken together, the considerations of the preceding paragraphs imply that obtaining solutions to \eqref{euler1}--\eqref{euleriancirculationcond} with a lidded atmosphere for a given $\rho$ and $\gamma$, is equivalent to solving the following problem:  Find $(\psi, \eta, Q)$ such that $\psi^{(i)}_y < 0$ in $\overline{\Omega^{(i)}}$, and 
\be \left\{ \begin{array}{ll} \Delta \psi + \gamma(\psi) = 0, & \textrm{in } \Omega, \\
\displaystyle \jump{|\nabla \psi|^2} + 2g \jump{\rho} \left(\eta+d\right)  - Q = 0, & \textrm{on } y = \eta(x), \\
\psi =  0 & \textrm{on } y = \ell, \\
\psi = -p_1, & \textrm{on } y = \eta(x), \\
\psi = -p_0, & \textrm{on } y = -d.\end{array} \right. \label{psieq} \ee
For any such solution, the relative circulation will then be given by \eqref{streamfunctioncirculationcond}.  We shall forestall a detailed discussion of the unbounded atmosphere case until \S \ref{unboundedidealsection} and \S\ref{unboundedshearsection}

Note that we have defined $\Omega$ so that it does not include the interface.  Thus \eqref{psieq} can be thought of as two separate elliptic problems in the domains $\Omega^{(1)}$ and $\Omega^{(2)}$, that must then be matched along the interface according to the jump condition \eqref{capgravdefQ}.  The advantage of this view point is that, while this a free boundary problem, the coefficients of the elliptic equation are all smooth.  Another way to proceed is to pose \eqref{psieq} in a weak sense on $\overline{\Omega}$, with the jump condition being represented as a measure supported on $\mathcal{I}$.  This conceptualization allows us to understand the matching procedure in the more conventional framework of elliptic problems with non-smooth coefficients, for which there is a great deal of theory.  As stated, we will be taking the first view --- namely, that \eqref{psieq} is two elliptic problems matched at the interface --- but occasionally will make use of the second view to derive some compactness properties of the corresponding operator.

Finally, let us discuss the circulation for the reformulated
problem.  Since $\mathscr{S}(p)$ is a level set of $\psi$, and we
have
\[ (u,v) = \frac{1}{\sqrt{\rho}} \nabla^\perp \psi + (c,0),\]
the circulation in the air is given by 
\[ \Gamma(p) = L c - 
\int_{\mathscr{S}(p)} \frac{1}{\sqrt{\rho(p)}} |\nabla \psi| \, d
\mathcal{H}^1,\] where $\mathcal{H}^1$ denotes one-dimensional
Hausdorff measure, 
which is equal to the arc length of the interface. 
It will be more convenient to consider the quantity \be
\Gamma_{\textrm{rel}}(p) := \frac{1}{L} \int_{\mathscr{S}(p)}
|\nabla \psi| \,
d\mathcal{H}^1,\label{streamfunctioncirculationcond} \ee which we
call the \emph{(pseudo) relative circulation}.  $\Gamma$ and
$\Gamma_{\textrm{rel}}$ are then related according to the
equation \be \Gamma_{\textrm{rel}}(p) = \sqrt{\rho(p)}
(c- \frac{\Gamma(p)}{L}).\label{defGammarel} \ee The advantage
of considering $\Gamma_{\textrm{rel}}$ in place of $\Gamma$ is
merely that it is simpler to express in terms of $\psi$, while
being equivalent for a specified $\rho$, $c$, and $L$.  

\subsection{Height equation formulation}

The main difficulty that remains in \eqref{psieq} is that the
domain $\Omega$ is an unknown.  Absent stagnation, this can be
rectified by considering a change of variables $(x, y) \mapsto
(q,p)$, where
\[ q := x, \qquad p: = -\psi(x,y).  \] This procedure is known
variously as the semi-Lagrangian transformation   
or the Dubreil-Jacotin transformation.  The
effect is to map a single period of $\Omega$ into a union of
rectangles $D = D^{(1)} \cup D^{(2)} \subset \mathbb{R}^2$, since
$\partial \Omega$ is mapped to the sets $\{ p = 0 \} \cup \{ p =
p_0\}$.  Note that, by definition, $\{ p = p_1\}$ is the image of
$\mathcal{I}$ under the transformation.  The image of the air
region $\Omega_1$ is thus
\[ D_1 := \{ (q,p) \in D : 0 < q < L, \, p_1 < p < 0\},\]
 while the water region is mapped to 
 \[ D_2 := \{ (q,p) \in D : 0 < q < L, \, p_0 < p < p_1\}.\]  
With that in mind, we put
\[ T := \{ p = 0\}, \qquad I := \{ p = p_1\}, \qquad B := \{ p= p_0\}.\]

Let $h = h(q,p)$ be the height above the bed of the point with $x = q$ and lying on the streamline $\{\psi = -p\}$, 
\[ h(q,p) := \mathfrak{y}(q,p) + d,\]
where $\mathfrak{y} = \mathfrak{y}(q,p)$ is the vertical variable $y$ in the new coordinates.  More explicitly, it is the unique solution to 
\[ \psi(q, \mathfrak{y}(q,p)) = -p,\]
the existence of which is guaranteed by the absence of stagnation.  By adopting the semi-Lagrangian coordinates, we are strongly exploiting the fact that there are no critical layers in the flow.

Equation  \eqref{psieq} can be reformulated as an equivalent problem for $h$: Find  $(h,Q)$ with $h$ even in $q$, $h_p > 0$ in $D$, and satisfying the height equation
\be \left \{ \begin{array}{lll}
(1+h_q^2)h_{pp} + h_{qq}h_p^2 - 2h_q h_p h_{pq} = -h_p^3 \gamma(-p), & \textrm{in } D_1 \cup D_2, \\
& & \\
\jump{\frac{\displaystyle 1+h_q^2}{\displaystyle h_p^2}} +2g\jump{\rho} h- Q  = 0, & \textrm{on } p = p_1, \\
& & \\
h = 0, & \textrm{on } p = p_0, \\
h = \ell + d(h), & \textrm{on } p = 0 \end{array}\right. \label{heighteq} \ee
Here $d$ is the depth operator
\[ d(h) := \frac{1}{L} \int_{-L/2}^{L/2} h(q,p_1)\, dq.\]
Note that we do not specify the value of $d$ in advance, rather
the correct value of $d(h)$  emerges self-consistently from the equations.
These equations  can be found by applying the procedure as in
\cite{walsh2009stratified} to obtain the interior equation, and
using the following change of variables formulas to reformulate
the jump condition:
\[ h_q = \frac{v}{u-c}, \qquad h_p =
\frac{1}{\sqrt{\varrho}(c-u)}, \qquad v = 
-\frac{h_q}{\sqrt{\varrho} h_p}, \qquad u = c- \frac{1}{\sqrt{\varrho} h_p},\]
\[ \partial_x = \partial_q - \frac{h_q}{h_p} \partial_p,
\qquad \partial_y = \frac{1}{h_p} \partial_p.\] 
Note that the slope $h_q$ is continuous across the interface so
could also be extracted from the jump condition on $p=p_1$.  
The relative circulation can be calculated from $h$ by
\[ \Gamma_{\textrm{rel}}(p) = \displaystyle \frac{1}{L} \int_{-L/2}^{L/2}
{\frac{1+h_q^2}{h_p}} \, dq,
\qquad \textrm{for } p_1 \leq p \leq 0.\]
%
Lastly, let us set down some notation and describe the regularity of the solutions we wish to study.  For any $k \in \mathbb{N}$, $\alpha \in (0,1)$, and smooth region $\mathcal{R} \subset \mathbb{R}^2$,  the space $C_{\textrm{per}}^{k+\alpha}(\overline{\mathcal{R}})$ is defined as the set of $C^{k+\alpha}(\overline{\mathcal{R}})$  that are $L$-periodic and even in their first coordinate.    We are seeking smooth solutions to the problems enumerated above.  Specifically, we look for solutions to the  Euler problem of class $\mathscr{S}$, the stream function problem of class $\mathscr{S}^\prime$, and the height equation of class $\mathscr{S}^{\prime\prime}$ defined as follows
\begin{align*} (u,v,\varrho, \eta) \in \mathscr{S} &:= \left(C_{\textrm{per}}^{\alpha}(\overline{\Omega}) \cap C_{\textrm{per}}^{1+\alpha}(\overline{\Omega} \setminus \mathcal{I}) \right)^3 \times  C_{\textrm{per}}^{2+\alpha}(\mathbb{R}), \\
(Q, \psi, \eta) \in \mathscr{S}^\prime & := \mathbb{R} \times (C_{\textrm{per}}^{1+\alpha}(\overline{\Omega_1}) \cap C_{\textrm{per}}^{1+\alpha}(\overline{\Omega_2}) \cap C_{\textrm{per}}^{\alpha}(\overline{\Omega}) \cap C_{\textrm{per}}^{2+\alpha}(\overline{\Omega} \setminus \mathcal{I}) ) \\
& \qquad  \times  C_{\textrm{per}}^{2+\alpha}(\mathbb{R}), \\
(Q, h) \in \mathscr{S}^{\prime\prime} & := \mathbb{R} \times \left(C_{\textrm{per}}^{1+\alpha}(\overline{D_1}) \cap C_{\textrm{per}}^{1+\alpha}(\overline{D_2}) \cap C_{\textrm{per}}^{\alpha}(\overline{D}) \cap C_{\textrm{per}}^{2+\alpha}(\overline{D} \setminus I)\right). \end{align*}
Put more plainly, we want, e.g., solutions to the height equation to be $C^{\alpha}$ in the whole domain, $C^{1+\alpha}$ up to the interface, and of class $C^{2+\alpha}$ away from the interface.   The regularity of the other quantities is a direct consequence of that choice.  

\begin{lemma}[Equivalence] \label{equivalencelemma} The following statements are equivalent.
\begin{romannum}
\item There exists a solution of class $\mathscr{S}$ to the steady stably stratified Eulerian problem \eqref{euler2}--\eqref{euleriancirculationcond} without stagnation.
\item  There exists a solution of class $\mathscr{S}^\prime$ to the streamfunction problem \eqref{psieq}.
\item There exists  a solution of class $\mathscr{S}^{\prime\prime}$ to the height equation problem \eqref{heighteq}.
\end{romannum}
\end{lemma}

\begin{proof} This lemma is routine.  See, for example, \cite[Lemma 2.1]{constantin2004exact} or \cite[Lemma 2.1]{walsh2009stratified}.  One point worth mentioning is that, while in our discussion of the formulation we only stated that $\psi$ was continuous across the boundary, as a solution of \eqref{psieq} it must in fact be of class $C_{\textrm{per}}^\alpha(\overline{\Omega})$ by elliptic regularity (see Theorem \ref{ellipticregtheorem}).    \qquad \end{proof}

\section{Local bifurcation for irrotational gravity
  waves} \label{liddedidealsection} We begin by considering the
simplest case where the flow in both the air and water regions is
irrotational ($\gamma \equiv 0$). In this setting we need only
consider the value of $\Gamma_{\textrm{rel}}$ on the lid.  That
is, the presence of wind 
here we interpret as having a nonzero $\Gamma_{\textrm{rel}}$ on
$p = 0$.  Without loss, therefore, in this section we redefine
$\Gamma_{\textrm{rel}}$ to be a positive constant.  The height
equation simplifies to: \be \left \{ \begin{array}{lll}
    (1+h_q^2)h_{pp} + h_{qq}h_p^2 - 2h_q h_p h_{pq} = 0 & \textrm{in } D_1 \cup D_2, \\
    & & \\
    \jump{\frac{\displaystyle 1+h_q^2}{\displaystyle h_p^2}} + 2g\jump{\rho} h- Q  = 0, & \textrm{on } p = p_1, \\
    & & \\
    h = 0, & \textrm{on } p = p_0, \\
    h = \ell + d(h), & \textrm{on } p = 0, \end{array}
\right. \label{idealheighteq} \ee and the relative circulation on
the lid for a solution $h$ is given by
\[ \displaystyle \frac{1}{L} \int_{-L/2}^{L/2} h_p(q,0)^{-1}\, dq = \Gamma_{\textrm{rel}}.\]
This follows from the fact that $h_q \equiv 0$ on the lid.  

Our main theorem on this topic is the following.  
\begin{theorem}[Local bifurcation for irrotational
  flows] \label{ideallocalbifurcation} Consider the existence of
  steady waves with a lidded atmosphere and irrotational flow in
  the air and water regions where $\Gamma_{\textrm{rel}}$ is
  given by \be |p_1| = \Gamma_{\textrm{rel}}
  \ell \label{compatibilitycond}.\ee Assume that the following
  ideal local bifurcation condition holds \be -g\jump{\rho} +
  \Gamma_{\mathrm{rel}}^2
  \coth{(\frac{p_1}{\Gamma_{\mathrm{rel}}})} <
  0. \tag{ILBC} \label{ideallbc} \ee Then there exists a
  continuous curve of non-laminar solutions to the height
  equation for irrotational flow \eqref{idealheighteq}  
\[ \mathcal{C}_{\mathrm{loc}}^{\prime\prime} = \{ (Q(s),h(s))  : |s| < \epsilon \} \subset \mathscr{S}^{\prime\prime}, \]
for $\epsilon > 0$ sufficiently small, such that $(Q(0), h(0)) = (Q(\lambda^*), H(\lambda^*))$, and, in a sufficiently small neighborhood of $(Q(\lambda^*), H(\lambda^*))$ in $\mathscr{S}^{\prime\prime}$, $\mathcal{C}_{\mathrm{loc}}^\prime$ comprises all non-laminar solutions. 
\end{theorem}

\begin{remark}  Using the equivalence of the three formulations, the theorem can be stated in terms of the original Eulerian problem as follows.  Consider the existence of steady waves with a lidded atmosphere and irrotational flow in the air and water regions.  Fix the (pseudo) volumetric mass flux in the air region to be $p_1$, the height of the lid to be $\ell$; this forces any laminar flow to have (pseudo) relative circulation $\Gamma_{\textrm{rel}}$ defined by \eqref{compatibilitycond}.  If, in addition, the ideal local bifurcation condition \eqref{ideallbc} is satisfied, then there is a corresponding continuous curve 
\[ \mathcal{C}_{\mathrm{loc}} = \{ (Q(s), u(s), v(s), \varrho(s), P(s), \eta(s)) : |s| < \epsilon \}\]
of small amplitude solutions to the Eulerian problem for an ideal fluid, which likewise captures all non-laminar solutions in a sufficiently small neighborhood of the point of bifurcation.  

It is also worth mentioning that, in contrast to the existence theory developed in \S\ref{liddedvorticalsection}, hypotheses \eqref{compatibilitycond} and \eqref{ideallbc} are both necessary and sufficient for local bifurcation to occur. 
\end{remark}

The proof of this is result is developed over the next several subsections.  

\subsection{Laminar flows}
A laminar flow is one in which the free surface is unperturbed, meaning that $\eta \equiv 0$, and where the streamlines are parallel to the bed.  In terms of the height equation formulation, this entails a solution with the ansatz $h = H(p)$, and $H(0) = d(H)$.  For such solutions, the PDE in \eqref{idealheighteq} reduces to an ODE: 
\be \label{ideallaminarprob} \left\{\begin{split} H_{pp} &= 0, \qquad \textrm{in } (p_0, p_1) \cup (p_1, 0),  \\
\jump{H_p^{-2}} + 2g\jump{\rho}H - Q &= 0, \qquad \textrm{on }  p = p_1, \\
H(0) &= \ell + d(H), \\
H(p_0) &= 0.  \end{split} \right. \ee
This problem can be easily solved explicitly, which leads to the following lemma.

\begin{lemma}[Laminar flow] \label{ideallaminarlemma}  For a fixed $p_0, p_1, \ell,$ and  $\jump{\rho}$, if $\Gamma_{\mathrm{rel}}$ is given by \eqref{compatibilitycond}, then there exists a one-parameter family of solutions $\{ (H(\cdot;\lambda), Q(\lambda)): \lambda > 0\}$ to the laminar flow equation \eqref{ideallaminarprob} with $H_p > 0$ and where each solution has relative circulation $\Gamma_{\mathrm{rel}}$ on the lid.  Explicitly,
\be H(p; \lambda) = \left\{\begin{array}{ll} \displaystyle \frac{p}{\Gamma_{\textrm{rel}}} + \ell + \frac{p_1 - p_0}{\lambda} , &  p_1 < p < 0  \\ & \\
\\ \displaystyle\frac{p-p_0}{\lambda}, & p_0 < p < p_1 \end{array} \right. \label{ideallaminarformula} \ee
and
\be Q(\lambda) = \frac{2g\jump{\rho}(p_1-p_0)}{\lambda} + \Gamma_{\textrm{rel}}^2 - \lambda^2.  \label{defideallaminarQ} \ee
In particular, the depth of the fluid at parameter value $\lambda$ is 
\be   d(\lambda) = \frac{\lambda}{p_1 - p_0}, \label{ideallaminardepth} \ee
and the width of the corresponding channel is
\be W(\lambda) :=  \ell + d(\lambda) = \ell + \frac{\lambda}{p_1 - p_0}. \label{ideallaminarwidth} \ee
\end{lemma} 
\begin{remark}  Let us make a few comments on this lemma.  \begin{remunerate}
\item  The proof of this lemma is straightforward.  Notice that \eqref{compatibilitycond} is required to ensure that the resulting solution is continuous across the interface.  Thus it is a necessary condition for the existence of laminar flows.  

\item Writing the corresponding solution in Eulerian form gives $(u,v) = (U(y), 0)$, where
\[ U(y; \lambda) =  \left\{\begin{array}{ll} \displaystyle c- \frac{\Gamma_{\textrm{rel}}}{\sqrt{\rho_{\textrm{air}}}}  , &  0< y < \ell  \\ \displaystyle c- \frac{\lambda}{\sqrt{\rho_{\textrm{water}}}} & -d< y < 0 \end{array} \right. \]
Thus the parameter $\lambda$ essentially dictates the relative speed in the water region while $\Gamma_{\textrm{rel}}$ dictates the relative speed in the air.

\item Differentiating \eqref{defideallaminarQ} in $\lambda$, it
  is clear that $\lambda \mapsto Q(\lambda)$ is concave and has a
  unique maximum 
at $\lambda = \lambda_0$, where
\be \lambda_0^3 = - g\jump{\rho} (p_1 - p_0).  \label{deflambda0}\ee \end{remunerate}   \end{remark}

\subsection{Linearized problem}
Fixing $\lambda > 0$, we next linearize the full height equation problem 
\eqref{idealheighteq} around $(H(\cdot;\lambda), Q(\lambda))$, which results in the following:
\be \left \{ \begin{array}{lll}
m_{pp} + m_{qq}H_p^2 = 0, & \textrm{in } D_1 \cup D_2, \\
\jump{H_p^{-3} m_p} = g\jump{\rho}m, & \textrm{on } p = p_1, \\ 
m =  0 & \textrm{on } p = p_0, \\
m - d(m) = 0 & \textrm{on } p = 0. \end{array} \right. \label{ideallinearproblem} \ee

For simplicity, we specialize to $L= 2\pi$;  the general case can be approached via rescaling.  Now, since we seek solutions that are $2\pi$-periodic and even in $q$, it is natural to take $m$ to have the ansatz $m(q,p) = M(p)\cos{(nq)}$, for some $n \geq 0$.  Inserting this into \eqref{ideallinearproblem} we see immediately that, for $n \geq 1$, $M$ must satisfy the following:
\be\left\{ \begin{array}{ll}
 a_\lambda^2 M_{pp} = n^2 M   & \textrm{in } (p_0, p_1) \cup (p_1, 0) \\ 
\jump{a_\lambda^3 M_p} = g \jump{\rho}M & \textrm{on } p = p_1, \\
M  = 0 & \textrm{on } p = p_0, \\
M = 0 &  \textrm{on } p = 0. \end{array}\right. \label{idealsturmliouvilleode} \tag{$P_n$}  \ee
Here we are denoting
\[ a_\lambda(p)  := H_p(p; \lambda)^{-1} = \left\{\begin{array}{ll}
 \Gamma_{\textrm{rel}} & p_1 < p < 0, \\   \lambda & p_0 < p < p_1. \end{array}\right. \]
On the other hand, when $n  = 0$, the equation in the interior are the same as above, but the depth operator $d$ does not vanish, meaning that the boundary condition is nonlocal:
\be\left\{ \begin{array}{ll}
 a_\lambda^2 M_{pp} = 0   & \textrm{in } (p_0, p_1) \cup (p_1, 0) \\ 
\jump{a_\lambda^3 M_p} = g \jump{\rho}M & \textrm{on } p = p_1, \\
M  = 0 & \textrm{on } p = p_0, \\
M(0)= M(p_1) . & \end{array}\right. \label{idealsturmliouvilleode0} \tag{$P_0$}  \ee  

Solving \eqref{idealsturmliouvilleode} for any value of $n \geq 1$ will produce a $2\pi/n$-periodic solution to the linearized problem.  The next lemma gives the relation between the wavelength, circulation, and the parameter value $\lambda$ (which, recall, is associated with the speed in the water region).   This can be seen as a form of dispersion relation.  

\begin{lemma} \label{dispersionlemma} For each $n \geq 1$, there exists a nontrivial solution $M$ to \eqref{idealsturmliouvilleode} if and only if $\lambda = \lambda_n$, where $\lambda_n$ satisfies
\be \frac{g \jump{\rho}}{n}  = \Gamma_{\mathrm{rel}}^2 \coth{(\frac{np_1}{\Gamma_{\mathrm{rel}}})} - (\lambda_n^*)^2 \coth{(\frac{n(p_1-p_0)}{\lambda_n^*})}.  \label{idealdispersionrelation} \ee
Such a $\lambda_n^*$ will exist if and only if 
\be \frac{g\jump{\rho}}{n} - \Gamma_{\mathrm{rel}}^2 \coth{(\frac{np_1}{\Gamma_{\mathrm{rel}}})} < 0.  \label{ideallineardispersionsizecond} \ee
Indeed, if it exists, $\lambda_n$ is unique.  
If \eqref{idealdispersionrelation} (or, equivalently, \eqref{ideallineardispersionsizecond})  holds, the space of solutions is one-dimensional and spanned by 
\be M_n^*(p) := \left\{ \begin{array}{ll} 
\sinh{(\frac{n p}{\Gamma_{\textrm{rel}}})} & p_1 < p < 0, \\
\mu \sinh{(\frac{n (p - p_0)}{\lambda_n^*})} & p_0 < p < p_1,
\end{array} \right. \label{idealMnformula} \ee
where 
\be \mu = \mu(\lambda, n) := \frac{\displaystyle \sinh{(\frac{n p_1}{\Gamma_{\textrm{rel}}})}}{\displaystyle \sinh{(\frac{n(p_1-p_0)}{\lambda})}}. \label{idealdefmui} \ee
Finally, there 
are nontrivial solutions to \eqref{idealsturmliouvilleode0} if and only if $\lambda = \lambda_0$, where $\lambda_0$ is as in \eqref{deflambda0}.  
\end{lemma}
\begin{proof}  Fix $n \geq 1$ and consider \eqref{idealsturmliouvilleode}.   By the ODE satisfied by $M$, we have immediately that 
\[ M(p) = C_1 \exp{(n a_{\lambda}^{-1} p)} + C_2 \exp{(-n a_{\lambda}^{-1} p)},\]
where 
\[C_i = \left\{\begin{array}{ll} C_i^{(1)} & \textrm{ in } (p_1,0), \\
C_i^{(2)} & \textrm{ in } (p_0, p_1), \end{array}\right. \qquad \textrm{for } i = 1,2. \]
From the boundary conditions on the top and bottom, we have
\[ C_1^{(1)} = -C_2^{(1)}, \qquad C_1^{(2)} = -C_2^{(2)} \exp{(-\frac{2n p_0}{\lambda})}.   \]
Continuity of $M$ at the interface implies that
\[ C_{1}^{(1)} \sinh{(\frac{np_1}{\Gamma_\textrm{rel}})} = C_1^{(2)} \exp{(\frac{np_0}{\lambda})} \sinh{(\frac{n(p_1-p_0)}{\lambda})}.\]
Incorporating these observations, we can write $M$ in the simplified form 
\be \begin{split} M^{(1)}(p) &= C \sinh{(\frac{np}{\Gamma_\textrm{rel}})}, \\
M^{(2)}(p) & = \mu C \sinh{(\frac{n(p-p_0)}{\lambda})}, \end{split} \label{ideallineareigenfunctions} \ee
where $\mu = \mu(\lambda,n)$ is as defined by \eqref{idealdefmui}.  

Lastly, we must ensure that the jump condition at the interface is met.  We compute from the above expression for $M$ that 
\begin{align*} \jump{a^3_\lambda M_p} &= \left[ \Gamma_{\textrm{rel}}^3 M_p^{(1)} - \lambda^3 M_p^{(2)}\right]\bigg|_{p=p_1} \\
& =  C n \Gamma_{\textrm{rel}}^2 \cosh{(\frac{np_1}{\Gamma_\textrm{rel}})} \\
& \qquad - \mu C n \lambda^2 \cosh{(\frac{n(p_1-p_0)}{\lambda})}. \end{align*}
Equating this to $g\jump{\rho} M(p_1)$ found via \eqref{ideallineareigenfunctions} and simplifying yields
\begin{align*} g\jump{\rho} \sinh{(\frac{np_1}{\Gamma_{\textrm{rel}}})} &= n \Gamma_{\textrm{rel}}^2 \cosh{(\frac{np_1}{\Gamma_{\textrm{rel}}})}  \\
& \qquad -  \mu n \lambda^2 \cosh{(\frac{n(p_1-p_0)}{\lambda})}.  \end{align*}
Recalling the definition of $\mu$, this becomes
\[ \frac{g \jump{\rho}}{n}  = \Gamma_{\textrm{rel}}^2 \coth{(\frac{np_1}{\Gamma_{\textrm{rel}}})} - \lambda^2 \coth{(\frac{n(p_1-p_0)}{\lambda})},    \]
which is the stated dispersion relation \eqref{idealdispersionrelation}.  

Fix $n \geq 1$ and consider the map $\lambda \in \mathbb{R}^+ \mapsto \lambda^2 \coth(n(p_1-p_0)/\lambda) \in \mathbb{R}^+$.  Elementary calculus confirms that it is a strictly increasing and nonnegative.  Therefore, provided that \eqref{ideallineardispersionsizecond} holds, i.e., 
\[ \frac{g\jump{\rho}}{n} - \Gamma_{\textrm{rel}}^2 \coth{(\frac{np_1}{\Gamma_{\textrm{rel}}})} < 0,\]
there is a unique $\lambda = \lambda_n^*$ for which the dispersion relationship \eqref{idealdispersionrelation} is satisfied. On the other hand, if this inequality does not hold, then there will be no such $\lambda$ and thus no nontrivial solutions to the eigenvalue problem.

Next consider the zero-mode case $n = 0$.  Letting $m(q,p) = M(p)$ in \eqref{ideallinearproblem} we see that $M$ must be piecewise linear.  Moreover, the condition at $p =0$ implies 
\[ M(0) =  d(M) = M(p_1).\]
From this we infer that $M \equiv M(0)$ on the interval $[p_1,0]$.  Due to the boundary condition at $p_0$ and the piecewise linearity of $M$, we know
\[ M(p) =  M(p_1) \frac{p-p_0}{p_1 - p_0}, \qquad p \in [p_0, p_1]. \]
Using this to evaluate the jump condition reveals that
\[ -\frac{\lambda^3}{p_1 - p_0} M(p_1) = g \jump{\rho} M(p_1).\]
Thus there is a nontrivial zero-mode solution if and only if
\[ \lambda^3 = - g\jump{\rho} (p_1 - p_0) = \lambda_0^3. \] 
This completes the proof.  
 \qquad \end{proof}

One final technical point needs to be made:  since the laminar curve is parameterized by $\lambda$, and the solutions we seek depend on $Q$, we need to ensure that at the point of bifurcation $Q$ is an invertible function of $\lambda$.  This is demonstrated in the next lemma.

\begin{lemma} \label{idealloclambda0lemma} For each $n \geq 1$ such that \eqref{ideallineardispersionsizecond} holds,  $Q$ is an invertible function of $\lambda$ in a neighborhood of $\lambda_n^*$.  
\end{lemma}
\begin{proof}   Fix $n \geq 1$ satisfying \eqref{ideallineardispersionsizecond} and let $\lambda_n^*$ and $M_n^*$ be defined as in \eqref{idealdispersionrelation} and \eqref{idealMnformula}, respectively.  From the formula  \eqref{defideallaminarQ}, it is obvious that $Q$ is a strictly concave function of $\lambda$, and hence to prove the lemma it suffices to show that $\lambda_n^* \neq \lambda_0$, where $\lambda_0$ is the unique critical point of $Q$. 

Multiplying \eqref{idealsturmliouvilleode} by $M_n^*$ and integrating by parts, we obtain 
\[ \int_{p_0}^0 a_{\lambda_n^*}^3 (\partial_p M_n^*)^2 \, dp + n^2 \int_{p_0}^0 a_{\lambda_n^*} (M_n^*)^2 \, dp + g\jump{\rho} M_n^*(p_1)^2 = 0.\]
Upon regrouping terms, this becomes 
\be 0 > -n^2 \int_{p_0}^0 a_{\lambda_n^*} (M_n^*)^2 = g \jump{\rho} M_n^*(p_1)^2 + \int_{p_0}^0 a_{\lambda_n^*}^3 (\partial_p M_n^*)^2 \, dp.  \label{idealloclambda01} \ee
On the other hand, as $M_n^*(p_0) = 0$,  
\begin{align*} M_n^*(p_1)^2 = \left(\int_{p_0}^{p_1} (\partial_p M_n^*) \, dp \right)^2 & \leq \left( \int_{p_0}^{p_1} a_{\lambda_0}^3 (\partial_p M_n^*)^2 \, dp \right) \left( \int_{p_0}^{p_1} a_{\lambda_0}^{-3} \, dp \right) \\
& = \left( \int_{p_0}^{p_1} a_{\lambda_0}^3 (\partial_p M_n^*)^2 \, dp \right) \left( \frac{p_1 -p_0}{\lambda_0^3} \right) \\ 
& = -\frac{1}{g \jump{\rho}} \int_{p_0}^{p_1} a_{\lambda_0}^3 (\partial_p M_n^*)^2 \, dp.\end{align*}
From this it follows that 
\be 0 \leq g\jump{\rho} M_n^*(p_1)^2 + \int_{p_0}^{p_1} a_{\lambda_0}^3 (\partial_p M_n^*)^2 \, dp. \label{idealloclambda02} \ee
Inequalities \eqref{idealloclambda01} and \eqref{idealloclambda02} cannot be reconciled if $\lambda_0 = \lambda_n^*$, allowing us to conclude that this is never the case.  \qquad \end{proof}

\subsection{Proof of local bifurcation}

The objective of this section is to apply the theory of local bifurcation from simple eigenvalues to construct small amplitude (non-laminar) solutions to the wind wave problem, eventually culminating in Theorem \ref{ideallocalbifurcation}.  The machinery we employ is the classical work of Crandall and Rabinowitz, which, in the interest of readability, is included in the appendix as Theorem \ref{crandallrabinowitz}.

Our first task is to put our problem into the framework of Theorem \ref{crandallrabinowitz}.  One cosmetic difference is that we wish to bifurcate from the family of laminar solutions, whereas Theorem \ref{crandallrabinowitz} concerns bifurcation from solutions of the form $(\lambda, 0)$.  With that in mind, let $(h,Q)$ solve the height equation, and suppose $h(q,p) = H(\cdot; \lambda) + m(q,p)$ and $Q = Q(\lambda)$.  Then
\be \left \{ \begin{array}{lll}
(1+m_q^2)(m_{pp}+H_{pp}) + m_{qq}(m_p+H_p)^2 & \\
 - 2m_q (H_p +m_p) m_{pq} = 0 & \textrm{in } D_1 \cup D_2, \\
& & \\
\jump{\frac{\displaystyle 1+m_q^2}{\displaystyle (H_p+m_p)^2}} + 2g\jump{\rho} (m+H)- Q  = 0, & \textrm{on } p = p_1, \\
& & \\
m = 0, & \textrm{on } p = p_0, \\
m+H - \ell - d(m) - d(H) = 0, & \textrm{on } p = 0. \end{array} \right. \label{perturbedidealheighteq} \ee
This can be restated equivalently as
\[ \mathcal{F}(\lambda, m) = 0,\]
where $\mathcal{F} = (\mathcal{F}_1, \mathcal{F}_2, \mathcal{F}_3, \mathcal{F}_4) : \mathbb{R} \times X \to Y$ is defined by
\be \label{idealF} \begin{split} \mathcal{F}_1(\lambda, w) & := (1+(w_q^{(1)})^2)(w_{pp}^{(1)}+H_{pp}) + w_{qq}^{(1)}(w^{(1)}_p+H_p)^2  \\
 & \qquad - 2w_q^{(1)} (H_p +w_p^{(1)}) w_{pq}^{(1)}  \\
 \mathcal{F}_2(\lambda, w) & := (1+(w_q^{(2)})^2)(w_{pp}^{(2)}+H_{pp}) + w_{qq}^{(2)}(w^{(2)}_p+H_p)^2  \\
 & \qquad - 2w_q^{(2)} (H_p +w_p^{(2)}) w_{pq}^{(2)} \\
\mathcal{F}_3(\lambda, w) & := -\jump{\frac{\displaystyle 1+w_q^2}{\displaystyle (H_p+w_p)^2}} - 2g\jump{\rho} (w+H)|_I + Q \\
\mathcal{F}_4(\lambda, w) & := \left(w-d(w)+ H-d(H) - \ell\right)|_T. 
 \end{split} \ee
 Here, the Banach spaces $X$ and $Y = Y_1 \times Y_2 \times Y_3 \times Y_4$ are
 \[ X := \{ h \in C_{\textrm{per}}^{2+\alpha}(\overline{D} \setminus I) \cap C^{\alpha}_{\textrm{per}}(\overline{D}): h(p_0) = 0, \, h^{(i)} \in C_{\textrm{per}}^{1+\alpha}(\overline{D_i})\},\] 
 \[ Y_1 :=  C_{\textrm{per}}^{2+\alpha}(\overline{D_1} \setminus I) \cap C^{\alpha}_{\textrm{per}}(\overline{D_1}), \qquad Y_2 := C_{\textrm{per}}^{2+\alpha}(\overline{D_2} \setminus I) \cap C_{\textrm{per}}^{\alpha} (\overline{D_2})\]
 \[ Y_3 :=  C_{\textrm{per}}^{\alpha}(I), \qquad Y_4 := C_{\textrm{per}}^{2+\alpha}(T).\]
 Observe that, by Lemma \ref{ideallaminarlemma}, $\mathcal{F}(\lambda, 0) = 0$ for every positive $\lambda$.  In particular, we shall consider bifurcation from the lowest eigenvalue of the linearized problem found in the previous section.  We shall therefore assume that \eqref{ideallineardispersionsizecond} holds for $n = 1$ and denote $\lambda^* := \lambda_1^*$.  
 
 For later reference, we now record the Fr\'echet derivative of $\mathcal{F}$ with respect to $w$ at $(\lambda^*,0)$.
 \be \begin{split}  \label{idealFw}
 \mathcal{F}_{iw}(\lambda^*, 0)\varphi & = \left(\partial_p^2 + H_p^2 \partial_q^2\right) \varphi^{(i)} \qquad \textrm{for } i = 1,2, \\
 \mathcal{F}_{3w}(\lambda^*,0)\varphi & = 2\jump{H_p^{-3} \varphi_p} - 2g \jump{\rho} \varphi, \\
 \mathcal{F}_{4w}(\lambda^*,0)\varphi & = \left(\varphi - d(\varphi) \right)_T.
 \end{split} \ee 
\begin{lemma}[Null space] \label{idealnullspacelemma} The null space of $\mathcal{F}_{w}(\lambda^*, 0)$ is one-dimensional.
\end{lemma}
\begin{proof}  Let $\varphi \in \mathcal{N}(\mathcal{F}_w(\lambda^*,0))$ be given.  By evenness, we can express $\varphi$ via a cosine expansion:
\[ \varphi(q,p) = \sum_{n=0}^\infty \varphi_n(p) \cos{(nq)}.\]
It follows that, 
\[ \mathcal{F}_w(\lambda^*,0)(\varphi_n(p) \cos{(nq)}) = 0, \qquad n \geq 0.\]
Equivalently, we must have that $\varphi_n$ solves \eqref{idealsturmliouvilleode}.  By Lemma \ref{dispersionlemma} and the definition of $\lambda^*$, we know that $\varphi_1$ is nontrivial and that $\varphi_n$ vanishes identically for $n \neq 1$.  We have therefore shown that $\mathcal{N}(\mathcal{F}_w(\lambda^*,0))$ is one-dimensional, and spanned by $\varphi^* := \varphi_1$, the unique solution to \eqref{idealsturmliouvilleode} for $n = 1$.  
\qquad \end{proof}

Now that we have ascertained the dimension of the null space, the natural next step in showing that $\mathcal{F}$ is Fredholm of index $0$ is to prove that the range is the (weighted) orthogonal complement of the null space.   

\begin{remark}
For the case we are studying where the air is irrotational, this is not particularly difficult if one takes the following approach.  To study the solvability of $\mathcal{F}_w(\lambda^*,0) \varphi = \mathcal{A}$ for $\mathcal{A} \in Y$, we may project onto the individual modes by expanding $\varphi(q,p) = \sum_n \varphi_n(p) \cos{(nq)}$.  Doing so, we see that $\varphi_n$ solves an inhomogeneous version of the linearized problem \eqref{idealsturmliouvilleode}.  When $n \geq 1$, the problem can be viewed simply as two second-order ODEs with Dirichlet boundary conditions on the top and bottom, which must then be matched so that the jump condition on the interface is satisfied.  The fact that the range has codimension 1 will arise as from this matching, and classical existence theory for ODEs.  As before, the zero mode is must be dealt with using the fact for $\lambda_*$ there is no 0 eigenvalue.    This procedure gives solutions of class $C^2$ away from the interface; the proof that the solution is in $X$ follows from elliptic regularity, as described in Theorem \ref{ellipticregtheorem}.

In \S\ref{liddedvorticalsection}, however, we allow the atmosphere to be rotational, and one of the effects of the vorticity is to make the linearized problem corresponding to \eqref{idealsturmliouvilleode} onerous to solve explicitly.  With that in mind, it seems that a better adapted approach is to avoid separating variables, relying instad on purely PDE existence theory.  Done this way, the proof of the range lemma is nearly identical in both regimes, and so this is the manner in which we have chosen to present it here.    \end{remark}

Even using PDE methods, though, the zero mode is somewhat special, since it is only there where the nonlocal operator $d$ can be seen.  For that reason, we will still wish to begin by projecting elements of $X$ and $Y$ onto their zero modes.  We adopt the following notation,
\[ (\mathbb{P}g)(\cdot) := \frac{1}{2\pi} \int_{-\pi}^{\pi} g(q,\cdot) \, dq, \qquad \textrm{for any } g \in X, Y_1, Y_2, Y_3, Y_4, \textrm{ or } Y_5.\]

\begin{lemma}[Range] \label{idealrangelemma} $\mathcal{A} = (\mathcal{A}_1, \mathcal{A}_2, \mathcal{A}_3, \mathcal{A}_4) \in Y$ is in the range of $\mathcal{F}_w(\lambda^*,0)$ if and only if it satisfies the following orthogonality condition
\be \int\!\!\!\int_{D_1} a^3 \mathcal{A}_1 \varphi^* \, dq \, dp + \int\!\!\!\int_{D_2} a^3 \mathcal{A}_2 \varphi^* \, dq \, dp + \frac{1}{2} \int_{I} \mathcal{A}_3 \varphi^* \, dq + \int_T a^3 \mathcal{A}_4 \varphi_p^* \, dq = 0. \label{idealrangecondition} \ee  
\end{lemma}
\begin{proof}  We begin by demonstrating necessity.  Suppose that $\mathcal{A} \in \mathcal{R}(\mathcal{F}_w(\lambda^*,0))$.  We may therefore let $\varphi$ be given such that $\mathcal{F}_w(\lambda^*,0)\varphi = \mathcal{A}$. 
It follows that 
\be \begin{split}  
\left( a^3 \varphi^*, \mathcal{A}_1 \right)_{L^2(D_1)} + \left(a^3 \varphi^*, \mathcal{A}_2\right)_{L^2(D_2)}  & = \int\!\!\!\int_{D_1} a^3 \left( \varphi_{pp} + H_p^2 \varphi_{qq} \right) \varphi^* \, dq \, dp \\
& \qquad + \int\!\!\!\int_{D_2} a^3 \left( \varphi_{pp} + H_p^2 \varphi_{qq} \right)  \varphi^* \, dq \, dp \\
& = - \int\!\!\!\int_{D_1}  a^3 \varphi_{p}  \varphi_p^* \, dq \, dp \\
& \qquad - \int\!\!\!\int_{D_2}  a^3 \varphi_{p}  \varphi_p^* \, dq \, dp \\
& \qquad  - \int_I \jump{a^3 \phi_p \varphi^*} \, dq \\
& \qquad\qquad  + \int\!\!\!\int_{D_1 \cup D_2}  a \varphi  \varphi_{qq}^* \, dq \, dp.  \end{split} \label{idealrangecomputation} \ee
Here we have exploited periodicity and the fact that $\varphi^*$ vanishes identically on $T$.  Continuing with the computation,
\begin{align*} \left( a^3 \varphi^*, \mathcal{A}_1 \right)_{L^2(D_1)} + \left(a^3 \varphi^*, \mathcal{A}_2\right)_{L^2(D_2)} & = \int\!\!\!\int_{D_1 \cup D_2} \left(a^3 \varphi_{pp}^* + a  \varphi_{qq}^*\right) \varphi \, dq\, dp \\
& \qquad + \int_I \left(\jump{a^3 \varphi_p^* \varphi} - \jump{a^3 \varphi_p \varphi^*} \right) \, dq  \\
& \qquad - \int_T a^3 \mathcal{A}_4 \varphi_p^* \, dq - d(\varphi) \int_T a^3 \varphi_p^* \, dq \\
& = \int_I  g\jump{\rho} \varphi^* \varphi \, dq \\
& \qquad + \int_I \left(-\frac{1}{2} \mathcal{A}_3 - g\jump{\rho} \varphi \right) \varphi^*  \, dq \\
& \qquad\qquad - \int_T a^3 \mathcal{A}_4 \varphi_p^* \, dq. 
\end{align*}
Simplifying, we see that, indeed, the orthogonality relation \eqref{idealrangecondition} must be satisfied.  The proof of necessity is complete.

The next (more difficult step) is to show that \eqref{idealrangecondition} is sufficient.  Suppose now that $\mathcal{A}$ satisfies \eqref{idealrangecondition}, we wish to prove that there exists $\varphi \in X$ such that $\mathcal{F}_w(\lambda^*,0)\varphi = \mathcal{A}$.  First we consider the zero mode problem found by applying the projection $\mathbb{P}$ to the equation, which gives 
\be \label{idealrangezeromodeprob} \begin{split}  \overline{\mathcal{A}}_1 & =  \partial_p^2 \overline{\varphi} \qquad \textrm{for } p_1 < p < 0,  \\
\overline{\mathcal{A}}_2 & =  \partial_p^2 \overline{\varphi} \qquad \textrm{for } p_0 < p < p_1, \\
\overline{\mathcal{A}}_3 & = 2 \jump{a^3 \overline{\varphi}_p} - 2g \jump{\rho} \overline{\varphi}, \\
\overline{\mathcal{A}}_4 & = \left(\overline{\varphi} - d(\overline{\varphi})\right)_T, \\
\end{split} \ee
where 
\[ \overline{\varphi} := \mathbb{P} \varphi, \qquad \overline{\mathcal{A}}  = (\overline{\mathcal{A}}_1, \overline{\mathcal{A}}_2, \overline{\mathcal{A}}_3, \overline{\mathcal{A}}_4) := (\mathbb{P} \mathcal{A}_1, \mathbb{P} \mathcal{A}_2, \mathbb{P} \mathcal{A}_3, \mathbb{P} \mathcal{A}_4).\]
Hence, 
\be \begin{split} \overline{\varphi}^{(1)}(p) &= \int_{p_1}^p \int_{p_1}^r \overline{\mathcal{A}}_1(s) \, ds \, dr + C_1 p + C_2,\\
 \overline{\varphi}^{(2)}(p) &= \int_{p_0}^p \int_{p_0}^r \overline{\mathcal{A}}_2(s) \, ds \, dr + C_3 (p-p_0),\end{split} \label{idealrangezeromodesol} \ee
for some constants $C_1, C_2$ and $C_3$.  The condition on $T$ tells us that 
\be - C_1 p_1 = \overline{\mathcal{A}}_4 - \int_{p_1}^0 \int_{p_1}^p \overline{\mathcal{A}}_1(r) \, dr \, dp.\label{idealrangezeromodeTcond} \ee
Continuity across the interface implies
\be C_1 p_1 + C_2 - (p_1-p_0) C_3 = \int_{p_0}^{p_1} \int_{p_0}^p \overline{\mathcal{A}}(r) \, dr \, dp.  \label{idealzeromodecontcond} \ee
Lastly, the jump condition equates requires that
\be \Gamma_{\textrm{rel}}^3 C_1 - (\lambda^*)^3\left( \int_{p_0}^{p_1} \overline{\mathcal{A}}_2(p) \, dp + C_3 \right) - g \jump{\rho} (C_1 p + C_2) = \frac{1}{2}\overline{\mathcal{A}}_3. \label{idealzeromodejumpcond} \ee
Collecting together \eqref{idealrangezeromodeTcond}, \eqref{idealzeromodecontcond}, and \eqref{idealzeromodejumpcond}, an easy calculation reveals that unique solvability of the zero mode problem is equivalent to 
\[ -g\jump{\rho} (p_1-p_0) \neq (\lambda^*)^3.\] 
But this is just the statement $\lambda^* \neq \lambda_0$, which is proved in Lemma \ref{idealloclambda0lemma}.  We have therefore shown that the zero-mode problem has a unique solution.  

Returning to the question of solvability of the full problem, we observe that, in light of the previous analysis,  it suffices to assume 
\[ \mathcal{A} \in Y_0 := (1-\mathbb{P}) Y =  \{ \mathcal{B} \in Y : \mathbb{P}\mathcal{B} = 0\},\]
 and to solve in the space 
 \[ \varphi \in X_0 := (1- \mathbb{P}) X = \{ \phi \in X : \mathbb{P} \phi = 0\}.\]
 In fact, this means that we may take $d(\varphi) = 0$.  
 
 We shall approach the question of solvability incrementally.  Fix $\epsilon > 0$, and define $\mathcal{L}^{(\epsilon)}:X_0 \to Y_0$ by
\[ \mathcal{L}^{(\epsilon)} = ( \epsilon - \mathcal{F}_{1w}(\lambda^*,0), \, \epsilon - \mathcal{F}_{2w}(\lambda^*,0), \,  -\mathcal{F}_{3w}(\lambda^*,0), \,  -\mathcal{F}_{4w}(\lambda^*,0)).\]
First consider  the approximate problem:
\be \mathcal{L}^{(\epsilon)} \varphi^{(\epsilon)} = \mathcal{A}.  \label{idealrangeapproxprob} \ee 

{\em Claim 1.}  For a sequence of $ \epsilon > 0$ tending to $0$, there exists a unique solution $\varphi^{(\epsilon)}$ to \eqref{idealrangeapproxprob}.  To see this, note that by a method of continuity argument, the solvability of \eqref{idealrangeapproxprob} is equivalent to that of the equation 
\[ \widetilde{\mathcal{L}^{(\epsilon)}} \varphi^{(\epsilon)} = \mathcal{A},\]
where
\[ \widetilde{\mathcal{L}^{(\epsilon)}}_i := \mathcal{L}^{(\epsilon)}_i, \textrm{ for } i \neq 3, \qquad \widetilde{\mathcal{L}^{(\epsilon)}}_3 \varphi^{(\epsilon)}  := 2 \jump{a^3 \varphi^{(\epsilon)}_p}.\]
That is, we may safely ignore the zero-th order term on the interface.   Let $\xi \in C_{\textrm{per}}^{\alpha}(\overline{D}) \cap C_{\textrm{per}}^{2+\alpha}(\overline{D_1} ) \cap C_{\textrm{per}}^{2+\alpha}(\overline{D_2}) $ be any function that exhibits the following properties: 
\[ \xi|_B = 0, \qquad \jump{a^3 \xi_p} = -\mathcal{A}_3, \qquad \xi|_T  = -\mathcal{A}_4, \qquad \textrm{and}\qquad \mathbb{P} \xi = 0. \]
Then, the solvability of \eqref{idealrangeapproxprob} is equivalent to that of 
\be \widetilde{\mathcal{L}^{(\epsilon)}} \widetilde{\varphi}^{(\epsilon)} = \widetilde{\mathcal{A}},\label{idealrangeapproxprob2} \ee
where $\widetilde{\varphi}^{(\epsilon)} := \varphi^{(\epsilon)} - \xi$, and 
\[ \widetilde{\mathcal{A}} := (\mathcal{A}_1 - \mathcal{L}_1^{(\epsilon)} \xi, \, \mathcal{A}_2 - \mathcal{L}_2^{(\epsilon)} \xi, \, 0, \,  0).\]
The Fredholm solvability of \eqref{idealrangeapproxprob2} follows from linear elliptic theory (cf. Theorem \ref{ellipticregtheorem}); we therefore need only to prove that the homogeneous problem associated with $\widetilde{\mathcal{L}^{(\epsilon)}}$ has no nontrivial solutions.   Suppose that $\phi$ solves the homogeneous problem for some fixed $\epsilon > 0$.  By evenness and the fact that $\mathbb{P} \phi = 0$, we can expand $\phi$ in a cosine series of the form 
\[ \phi(q,p) = \sum_{n = 1}^\infty \phi_n(p) \cos{(nq)} .\]
Applying the operator $\widetilde{\mathcal{L}^{(\epsilon)}}$ to $\phi$ reveals that $\phi_n$ satisfies the eigenvalue problem
\[ \partial_p^2 \phi_n  = (\epsilon + n^2 H_p^2) \phi_n,  \qquad 
\jump{a^3 \partial_p \phi_n} =  \phi_n(p_0) = \phi_n(0) = 0.\]
This is a consequence of the fact that $d(\phi_n) = 0$, as $n \geq 1$.  This equation can be easily solved explicitly, and one can readily see that, for generic $\epsilon$ small, there are no nontrivial solutions for any $n \geq 1$.  We omit the details in the interest of brevity.  The first claim is proved.  \\

We may therefore let $\epsilon_n$ be a positive sequence converging to $0$ as in Claim 1, and consider the corresponding sequence of solutions to \eqref{idealrangeapproxprob}, call them $\{ \varphi^{(n)} \}$. \\
  
 {\em Claim 2.}  $\{ \varphi^{(n)} \}$ is bounded uniformly in $C_{\textrm{per}}^{1+\alpha}(\overline{D_1}) \cap  C_{\textrm{per}}^{1+\alpha}(\overline{D_2})$.   We argue by contradiction.  Suppose that $\{ \varphi^{(n)}\}$ is not bounded uniformly.  Possibly by passing to a subsequence, we may suppose that $\|\varphi^{(n)}\|_{C_{\textrm{per}}^{1+\alpha}(\overline{D_1})} + \|\varphi^{(n)}\|_{C_{\textrm{per}}^{1+\alpha}(\overline{D_2})}  \to \infty$.  Let $\phi^{(n)} := \varphi^{(n)}/ (\|\varphi^{(n)}\|_{C_{\textrm{per}}^{1+\alpha}(\overline{D_1})} + \|\varphi^{(n)}\|_{C_{\textrm{per}}^{1+\alpha}(\overline{D_2})})$.  Thus $\{ \phi^{(n)}\}$ has unit $C_{\textrm{per}}^{1+\alpha}(\overline{D_1}) \cap C_{\textrm{per}}^{1+\alpha}(\overline{D_2}) $-norm, and, by linearity, is a solution to 
 \[ \mathcal{L}^{(\epsilon_n)} \phi^{(n)} = \frac{\mathcal{A}}{\|\varphi^{(n)}\|_{C_{\textrm{per}}^{1+\alpha}(\overline{D_1})} + \|\varphi^{(n)}\|_{C_{\textrm{per}}^{1+\alpha}(\overline{D_2})}} =: \mathcal{A}^{(n)}.\]
 Since $\mathcal{A}^{(n)}$ converge to 0 in $C_{\textrm{per}}^{1+\alpha}(\overline{D_1}) \cap  C_{\textrm{per}}^{1+\alpha}(\overline{D_2})$, we may extract a subsequence $\{\phi^{(n_k)}\}$ converging to $\phi \in C^2(\overline{D} \setminus I) \cap C^1(\overline{D_1}) \cap C^1(\overline{D_2})$, a classical solution of
 \[ \mathcal{F}_w(\lambda^*,0) \phi = 0.\]
 This is achieved by using Schauder-type estimates for the approximate problem, and the compactness of the embedding of $C^{k+\alpha}$ into $C^k$ on bounded domains.  From Lemma \ref{idealnullspacelemma} it follows that $\phi = \nu \varphi^*$, for some $\nu \in \mathbb{R}$.  Finally, we note that it must be the case that $\|\phi\|_{C_{\textrm{per}}^{1+\alpha}(\overline{D_1})} + \|\phi\|_{C_{\textrm{per}}^{1+\alpha}(\overline{D_2})} = 1$, and hence $\nu \neq 0$. 
  
 Recall that, by definition the of $\varphi^{(n)}$,  we have 
 \[ a^3 \varphi^* \mathcal{L}^{(\epsilon_n)} \varphi^{(n)} = a^3 \varphi^* \mathcal{A}. \]
 Integrating over $D$, we obtain 
 \begin{align*}
 \left( a^3 \varphi^*, \mathcal{A}_1 \right)_{L^2(D_1)} + \left( a^3 \varphi^*, \mathcal{A}_2 \right)_{L^2(D_2)}
& = \int\!\!\!\int_{D_1} a^3 \varphi^* (\epsilon_n \phi^{(n)} - \phi_{pp}^{(n)} - H_p^2 \phi_{qq}^{(n)} ) \, dq\, dp \\
 & + \int\!\!\!\int_{D_2} a^3\varphi^* (\epsilon_n \varphi^{(n)} - \varphi_{pp}^{(n)} - H_p^2 \phi_{qq}^{(n)} ) \, dq\, dp \\
 & =  \epsilon_n \int\!\!\!\int_{D_1 \cup D_2}a^3 \varphi^* \varphi^{(n)}  \, dq \, dp \\
 & \qquad + \int_I \left( \jump{a^3 \varphi^{(n)}  \varphi^*_p} - \jump{a^3 \varphi_p^{(n)}  \varphi^*} \right) \, dq  \\
 & \qquad \qquad + \int_T a^3 \varphi_p^* \varphi^{(n)} \, dq \\
 & =  \epsilon_n \int\!\!\!\int_{D_1 \cup D_2}a^3 \varphi^* \varphi^{(n)}  \, dq \, dp \\
 & \qquad - \frac{1}{2} \int_I a^3 \mathcal{A}_3 \varphi^* \, dq - \int_T a^3 \mathcal{A}_4 \varphi_p^* \, dq.
\end{align*}
By the orthogonality condition \eqref{idealrangecondition}, all the terms involving $\mathcal{A}$ cancel, leaving only the statement that 
\[ \int\!\!\!\int_{D_1 \cup D_2} a^3 \varphi^* \varphi^{(n)} \, dq\, dp = 0 \qquad \textrm{for all } n \geq 1.\]
As an immediate consequence, we have 
\[ \int\!\!\!\int_{D_1 \cup D_2} a^3 \varphi^* \phi \, dq\, dp = 0, \]
which contradicts the fact that $\phi = \nu \varphi^*$.  This completes the proof of the second claim.  \\

Up to this point we have shown that there exists solutions $\{ \varphi^{(n)} \}$ to \eqref{idealrangeapproxprob} for a sequence of $\epsilon_n \to 0$ that are uniformly bounded in $C^{1+\alpha}$.  By passing to a subsequence, we have that there exists a weak solution to the original problem $\mathcal{F}_w(\lambda^*,0) \varphi = \mathcal{A}$.  Elliptic regularity ensures that the weak solution is actually strong, and, in particular, is an element of $X$ (cf. Theorem \ref{ellipticregtheorem}).  It follows that $\mathcal{A}$ is in the range.  The proof of the lemma is complete. \qquad \end{proof}

\begin{lemma}[Transversality] \label{idealtransverselemma} The following transversality condition holds:
\be \label{idealtranvscond} \mathcal{F}_{\lambda w}(\lambda^*,0) \varphi^* \not\in \mathcal{R}(\mathcal{F}_w(\lambda^*,0)). \ee
Here $\varphi^*$ denotes a generator of $\mathcal{N}(\mathcal{F}_w(\lambda^*,0))$.  
\end{lemma}
\begin{proof} In view of Lemma \ref{idealrangelemma}, it suffices to show that $\mathcal{A} := \mathcal{F}_{\lambda w}(\lambda^*,0) \varphi^*$ fails to satisfy the orthogonality condition \eqref{idealrangecondition}.  That is, if we put 
\[ \Xi := \int\!\!\!\int_{D_1} a^3 \mathcal{A}_1 \varphi^* \, dq \, dp + \int\!\!\!\int_{D_2} a^3 \mathcal{A}_2 \varphi^* \, dq \, dp + \frac{1}{2} \int_{I} \mathcal{A}_3 \varphi^* \, dq + \int_T a^3 \mathcal{A}_4 \varphi_p^* \, dq,\]
we must prove $\Xi \neq 0$. To do this, we first compute
\begin{align*}
\mathcal{F}_{1\lambda w}(\lambda^*,0) \varphi^* & = 0, \\
 \mathcal{F}_{2\lambda w}(\lambda^*,0) \varphi^* & = -\frac{2}{(\lambda^*)^3}  \varphi_{qq}^*, \\
 \mathcal{F}_{3\lambda w}(\lambda^*,0) \varphi^* & =  \left(-3(\lambda^*)^2 (\varphi_p^*)^{(2)} \right)_I, \\
 \mathcal{F}_{4\lambda w}(\lambda^*,0) \varphi^* & = 0.
\end{align*}
By the equation satisfied by $\varphi^*$, we see that 
\[ -\partial_q^2 \varphi^* = \varphi^* = a^2 \partial_p^2 \varphi^*.\]
Thus, 
\[ (\lambda^*)^2  ( \varphi^* \varphi_p^*)_p = (\lambda^*)^2 (\varphi_p^*)^2 + (\varphi^*)^2, \qquad \textrm{in } D_2.\]
From this we deduce,
\be \int\!\!\!\int_{D_2} a^3 \mathcal{A}_2 \varphi^* \, dq \, dp  = 2 \int\!\!\!\int_{D_2} (\varphi^*)^2  \, dq\, dp, \label{idealtrans1} \ee
and 
\be \begin{split} \frac{1}{2} \int_I \mathcal{A}_3 \varphi^* \, dq  &= -\frac{3(\lambda^*)^2}{2} \int_I (\varphi_p^*)^{(2)} \varphi^* \, dq \\
& = -\frac{3 (\lambda^*)^2}{2} \int\!\!\!\int_{D_2} (\varphi_p^*)^2 \, dq\, dp - \frac{3}{2} \int\!\!\!\int_{D_2} (\varphi^*)^2 \,dq \, dp.
\end{split} \label{idealtrans2} \ee
In light of \eqref{idealtrans1}--\eqref{idealtrans2}, and the calculated values of the remaining components of $\mathcal{A}$, 
\be \Xi = \frac{1}{2} \int\!\!\!\int_{D_2} (\varphi^*)^2 \, dq \, dp - \frac{3}{2} (\lambda^*)^2 \int\!\!\!\int_{D_2} (\varphi_p^*)^2 \, dq\, dp. \label{idealtrans3} \ee

Recall that $(\varphi^*)^{(2)}$ takes the form 
\[ \varphi^*(p) = \mu \sinh{(\frac{p-p_0}{\lambda^*})} \qquad \textrm{in } D_2,\]
for an explicit constant $\mu$.  Hence,
\[ (\lambda^*)^2 \varphi_p^*(p)^2 = \mu^2 \cosh^2{(\frac{p-p_0}{\lambda^*})} \qquad \textrm{in } D_2,\]
and thus
\[ \varphi^*(p)^2 - 3(\lambda^*)^2 \varphi_p^*(p)^2 = -\mu^2 \left(1 + \cosh^2{(\frac{p-p_0}{\lambda^*})} \right) \qquad \textrm{in } D_2.  \]
From this we conclude $\Xi < 0$.  \qquad \end{proof}

With these lemmas, Theorem \ref{ideallocalbifurcation} becomes a simple consequence of the Crandall--Rabinowitz bifurcation theorem.

\pfthm{\ref{ideallocalbifurcation}} Assuming \eqref{compatibilitycond} and \eqref{ideallbc}, we are justified in defining $\mathcal{F}$ and $\lambda^*$ as in \eqref{idealF} by Lemma\ref{ideallaminarlemma} and Lemma \ref{dispersionlemma}, respectively.  To complete the proof, we need only confirm that the hypotheses of Theorem \ref{crandallrabinowitz} are satisfied.    But parts (i) and (ii) of the Crandall--Rabinowitz theorem are obviously true.  Lemma \ref{idealnullspacelemma} and Lemma \ref{idealrangelemma} together give (iii), while (iv) was proved in Lemma \ref{idealtransverselemma}. \qquad \endproof

\section{Local bifurcation with vorticity in the atmosphere} \label{liddedvorticalsection}

Next consider the case where the density remains constant in each region, but we assume that the air region has a nontrivial vorticity strength function.  The height equation for this scenario is 
\be \left \{ \begin{array}{lll}
(1+h_q^2)h_{pp} + h_{qq}h_p^2 - 2h_q h_p h_{pq} = -\gamma(-p) h_p^3 & \textrm{in } D_1 \cup D_2, \\
& & \\
\jump{\frac{\displaystyle 1+h_q^2}{\displaystyle h_p^2}} + 2g\jump{\rho} h- Q  = 0, & \textrm{on } p = p_1, \\
& & \\
h = 0, & \textrm{on } p = p_0, \\
h = \ell + d(h), & \textrm{on } p = 0.\end{array} \right. \label{shearheighteq} \ee
Written in the new variables, the relative circulation becomes
\[ \Gamma_{\textrm{rel}} = \displaystyle \frac{1}{L} \int_{-L/2}^{L/2} \frac{1+h_q^2}{h_p} \, dq, \qquad\textrm{for } p_1 \leq p \leq 0.\]
Our main theorem is the following.

\begin{theorem}[Local bifurcation with atmospheric vorticity]  \label{shearlocalbifurcationtheorem} For given $p_1$, $\ell$, and $\gamma$, define $\Gamma_{\mathrm{rel}}$ by
\be \partial_p (\Gamma_{\mathrm{rel}}(p)^2) = 2\gamma(-p), \qquad \ell = \int_{p_1}^0 \frac{dp}{\Gamma_{\mathrm{rel}}(p)}, \label{shearcompatibilitycond} \ee
If the local bifurcation condition holds (cf. Definition \ref{lbcdef}) then there exists a continuous curve of non-laminar solutions to the height equation for irrotational flow in the water and vorticity strength function $\gamma$ in the air \eqref{shearheighteq}
\[ \mathcal{C}_{\textrm{loc}}^{\prime\prime} = \{ (Q(s),h(s)) \in \mathbb{R} \times X : |s| < \epsilon \}, \]
for $\epsilon > 0$ sufficiently small, such that $(Q(0), h(0)) = (Q(\lambda^*), H(\lambda^*))$, a laminar solution with (pseudo) relative circulation $\Gamma_{\mathrm{rel}}$, and, in a sufficiently small neighborhood of $(Q(\lambda^*), H(\lambda^*))$ in $\mathbb{R}\times X$, $\mathcal{C}_{\textrm{loc}}^\prime$ comprises all non-laminar solutions. \\
\end{theorem}
\begin{remark}  As in Theorem \ref{ideallocalbifurcation}, we can interpret the above statement in terms of the Eulerian formulation.  The resulting statement is as follows.  Fix the period to be $2\pi$,  (pseudo) volumetric mass flux in the air region $p_1$, lid height $\ell$, and with vorticity strength function $\gamma$ and define $\Gamma_{\mathrm{rel}}$ by \eqref{shearcompatibilitycond}.  If the local bifurcation condition holds, then there is a corresponding continuous curve 
  \[ \mathcal{C}_{\textrm{loc}} = \{ (Q(s), u(s), v(s),
  \varrho(s), P(s), \eta(s)) : |s| < \epsilon \}\] of small
  amplitude solutions to the Eulerian problem with an
  irrotational 
  water region and a vorticity strength function $\gamma$ in the
  air region, which likewise captures all non-laminar solutions
  in a sufficiently small neighborhood of the point of
  bifurcation.

Lemma \ref{shearsizeconditionlemma} provides an explicit size condition  \eqref{shearsizecond} under which the local bifurcation condition holds.  We have left it in the more abstract form here in order to give the most general statement.  On the other hand, the compatibility condition is, in fact, necessary for the existence of laminar solutions.  \end{remark}

\subsection{Laminar flows}  Consider the laminar flow problem where we seek solution to \eqref{shearheighteq} with the ansatz $h(q,p) = H(p)$.  Then the PDE simplifies to the following 
\be \label{shearlaminarprob} \left\{\begin{split} H_{pp} &= 0, \qquad \textrm{in } (p_0, p_1)  \\
H_{pp} &= -\gamma(-p) H_p^3, \qquad \textrm{in } (p_1,0) \\
\jump{H_p^{-2}} + 2g\jump{\rho}H - Q &= 0, \qquad \textrm{on }  p = p_1, \\
H(0) &= \ell + d(H), \\
H(p_0) &= 0.  \end{split} \right. \ee
Evaluating the relative circulation gives,
\[H_p^{-1} = \Gamma_{\textrm{rel}}, \qquad \textrm{in } [p_1,0].\]

Again, this is explicitly solvable, but there are some compatibility conditions that are necessary to ensure continuity across the interface.  

\begin{lemma}[Laminar lemma] \label{shearlaminarlemma}   If the compatibility condition \eqref{shearcompatibilitycond} is satisfied, then there exists a one-parameter family of solutions $\{ (H(\cdot;\lambda), Q(\lambda)): \lambda > 0\}$ to the laminar flow equation \eqref{shearlaminarprob} with $H_p > 0$ and where each member of the family has prescribed relative circulation $\Gamma_{\mathrm{rel}}$.  It has the explicit form
\be H(p; \lambda) = \left\{\begin{array}{ll} \displaystyle  \int_{p_1}^p \frac{dr}{\Gamma_{\mathrm{rel}}(r)} +  \frac{p_1 - p_0}{\lambda} , &  p_1 < p < 0  \\ & \\
\\ \displaystyle\frac{p-p_0}{\lambda}, & p_0 < p < p_1 \end{array} \right. \label{shearlaminarformula} \ee
and
\be Q(\lambda) = \frac{2g\jump{\rho}(p_1-p_0)}{\lambda} + 
\Gamma_{\mathrm{rel}}(p_1)^2 - \lambda^2.  \label{defshearlaminarQ} \ee
The depth of the fluid at parameter value $\lambda$ is 
\be   d(\lambda) = \frac{\lambda}{p_1 - p_0}, \label{shearlaminardepth} \ee
and the width of the corresponding channel is
\be W(\lambda) :=  \ell + d(\lambda) = \ell + \frac{\lambda}{p_1 - p_0}. \label{shearlaminarwidth} \ee
\end{lemma}
\begin{proof}  Examining the ODE in the air region reveals that 
\[ \partial_p \left(\frac{1}{2H_p^2}\right) = \gamma, \qquad \textrm{for } p \in (p_1, 0).\] 
This implies directly that $(\Gamma_{\textrm{rel}}^2)^\prime = 2\gamma.$   Integrating the equation $H_p^{-1} = \Gamma_{\textrm{rel}}$, we have
\[ H^{(1)}(p) = \int_{p_1}^p \frac{dr}{\Gamma_{\textrm{rel}}(r)} + C, \]
for some constants $C$.  Since $H_{pp}$ vanishes in the water region, and $H(p_0) = 0$,
\[ H^{(2)}(p) = \frac{p - p_0}{\lambda}, \]
for some constant $\lambda > 0$.  Continuity at the interface then implies that $C = (p_1 - p_0)/\lambda$.  Using this, and the fact that $d(H) = H(p_1)$, the condition on $T$ becomes
\[ \ell = \int_{p_1}^0 \frac{dr}{\Gamma_{\textrm{rel}}(r)},\]
which is the second part of the compatibility condition \eqref{shearcompatibilitycond}.

Lastly, we use the jump condition on the interface to determine $Q$ as a function of $\lambda$:
\[ Q = \jump{H_p^{-2}} + 2g \jump{\rho} H(p_1) = \Gamma_{\textrm{rel}}(p_1)^2 - \lambda^2 +2g\jump{\rho} \frac{p_1-p_0}{\lambda},\]
which is \eqref{defshearlaminarQ}. \qquad \end{proof}

\begin{remark} The dependence of $Q$ on $\lambda$ is essentially
  unchanged from the ideal case \eqref{defideallaminarQ}.  In
  particular, it is concave with a unique maximum 
  occurring at $\lambda_0$ as defined in \eqref{deflambda0}.
\end{remark}

\subsection{Linearized problem}  Proceeding as before, we seek to linearize the full height equation around the curve of laminar flows.  For a fixed $\lambda > 0$, this gives
\be \left \{ \begin{array}{lll}
(a^3 m_p)_p  + (am_{q})_q  = 0, & \textrm{in } D_1 \cup D_2, \\
\jump{a^3 m_p} = g\jump{\rho}m, & \textrm{on } p = p_1, \\ 
m =  0, & \textrm{on } p = p_0, \\
m - d(m) = 0, & \textrm{on } p = 0. \end{array} \right. \label{shearlinearproblem} \ee
Here we are again using the convention
\[ a = a(p;\lambda) = H_p(p; \lambda)^{-1} = \left\{ \begin{array}{ll} \Gamma_{\textrm{rel}}(p), & p_1 < p < 0, \\
\lambda, & p_0 < p < p_1.\end{array}\right.\]
Note that the linearization of the circulation identity is
\[ \displaystyle \frac{1}{2\pi} \int_{-\pi}^{\pi} a m_p \, dq = 0 \qquad \textrm{for } p_1 < p < 0,\]
and hence each of the solutions of the linearized problem will have no relative circulation in the air region.

We seek solutions with the ansatz $m(q,p) = M(p)\cos{(nq)}$.  First consider the case where $n = 0$, i.e., there is no $q$-dependence.  Then
\[  \left \{ \begin{array}{lll}
(a^3 M_p)_p  = 0, & \textrm{in } D_1 \\
M_{pp} = 0, & \textrm{in } D_2 \\
\jump{a^3 M_p} = g\jump{\rho}M, & \textrm{on } p = p_1, \\ 
M =  0, & \textrm{on } p = p_0, \\
M - d(M) = 0, & \textrm{on } p = 0. \end{array} \right. \]
Since $M$ is linear in $D_2$, the boundary condition at the bottom implies that it takes the form
\[ M^{(2)}(p) = \frac{p-p_0}{p_1-p_0} M(p_1), \qquad p \in [p_0, p_1].\]
On the other hand, since $M(0) = d(M) = M(p_1)$, we must have that $M_p$ vanishes at least once in the interior of $D_1$.  Since $a^3 M_p$ is constant in $D_1$, we conclude that $M_p^{(1)}$ must, in fact, be identically zero.  Thus the jump boundary condition states
\[ - \frac{\lambda^3}{p_1 - p_0} M(p_1) = g\jump{\rho} M(p_1).\] 
Just as in the irrotational case, we see that there can be a zero-mode solution if and only if $\lambda^3 = - g\jump{\rho} (p_1 - p_0)$, that is, $\lambda = \lambda_0$ (cf. Lemma \ref{dispersionlemma}).  

If we take $n \geq 1$, then \eqref{shearlinearproblem} becomes
\be\left\{ \begin{array}{ll}
 (a^3 M_{p})_p  = n^2 a M   & \textrm{in } (p_0, p_1) \cup (p_1, 0) \\ 
\jump{a_\lambda^3 M_p} = g \jump{\rho}M & \textrm{on } p = p_1, \\
M  = 0 & \textrm{on } p = p_0, \\
M = 0 &  \textrm{on } p = 0. \end{array}\right. \label{shearsturmliouvilleode} \tag{$P_n^\prime$} \ee
Notice that periodicity implies that $d(M) = 0$.  

We shall approach the problem of finding solutions to \eqref{shearsturmliouvilleode} using a variational method.  Define the Rayleigh quotient $\mathscr{R}$ by
\[ \mathscr{R}(\varphi; \lambda) := \frac{g \jump{\rho} \varphi(p_1)^2 + \int_{p_0}^0 a^3 \varphi_p^2 \, dp}{\int_{p_0}^0 a \varphi^2 \, dp}, \qquad \lambda > 0, \, \varphi \in \mathscr{A}, \]
where the admissible set 
 \[ \mathscr{A} := \{ \phi \in L^2([p_0, 0]) \cap H^1([p_0,p_1)) \cap H^1((p_1,0]) : \phi(0) = \phi(p_1) = 0\}.\]
 A straightforward argument then gives the following lemma. 
\begin{lemma} \label{variationalcharacterizationlemma} If, for fixed $\lambda > 0$, $\varphi$ is a critical point of $\mathscr{R}(\cdot; \lambda)$ and $\mathscr{R}(\varphi; \lambda) = -n^2$, for some $n \geq 1$, then $\varphi$ solves \eqref{shearsturmliouvilleode} for this value of $n$.  \end{lemma}

Using very basic estimates, we can show that for $\lambda$ sufficiently large, $\min_{\mathscr{A}}\mathscr{R}(\cdot; \lambda)$ will be greater than $-1$.  More generally, we can show that the minimum goes to $-\infty$ as $\lambda \to +\infty$.  This is the content of the next lemma. 

\begin{lemma} \label{boundRlemma} Let $a_{\textrm{min}}$ be the minimum value of $a$ on $[p_1,0]$ (which does not depend on $\lambda$).  Then, for each $n \geq 1$, if 
\[ \lambda^2 > a_{\textrm{min}}^2 -\frac{1}{2n} g \jump{\rho} ,\]
we have $\mathscr{R}(\varphi; \lambda) > -n^2$, for every $\varphi \in \mathscr{A}$.   
\end{lemma}  
\begin{proof}  Fix any $\lambda$ as in the hypothesis and let $\varphi \in \mathscr{A}$ be given.  Then
\begin{align*} \int_{p_1}^0 \left( a^3 \varphi_p^2 + n^2 a \varphi^2 \right) \, dp & \geq a_{\textrm{min}} \int_{p_1}^0 \left ( (a_{\textrm{min}} \varphi_p)^2 + (n \varphi)^2 \right)\, dp  \\
& \geq 2n a_{\textrm{min}}^2 \int_{p_1}^0 \varphi_p \varphi \, dp = -2n a_{\textrm{min}}^2 \varphi(p_1)^2. \end{align*}
On the other hand, since $a^{(2)} = \lambda$, 
\[ \int_{p_0}^{p_1} \left( a^3 \varphi_p^2 + n^2 a \varphi^2 \right) \, dp \geq 2 n \lambda^2 \varphi(p_1)^2.\]
Summing these together and recalling how we chose $\lambda$, we find
\[ \int_{p_0}^0 \left( a^3 \varphi_p^2 + n^2 a \varphi^2 \right) \, dp \geq (2n\lambda^2 - 2n a_{\textrm{min}}^2) \varphi(p_1)^2 > - g\jump{\rho}\varphi(p_1)^2.\]
Rearranging terms, this implies that $\mathscr{R}(\varphi; \lambda) > -n^2$.   \qquad \end{proof}

Let us define 
\[  \nu(\lambda) := \min_{\substack{\varphi \in \mathscr{A} \\ \varphi \not\equiv 0}}  \mathscr{R}(\varphi; \lambda).  \]
Then, if we can show that $\nu(\lambda_n^*) = -n^2$, for some $\lambda^*_n$, Lemma \ref{variationalcharacterizationlemma} guarantees the existence of a solution to \eqref{shearsturmliouvilleode}, and thereby a $2\pi/n$-periodic solution to the linearized problem.  We are most interested in showing that $-1$ is in the range of $\nu$.    The preceding lemma provides a lower bound for $\nu$ when $\lambda$ is sufficiently large.  In order to guarantee that $-1$ is in the range of $\nu$, therefore, we need only verify that $\nu(\lambda) < -1$ for some positive $\lambda$.   This will not be true in general, and so we are forced to make it a hypothesis.

\begin{definition}  \label{lbcdef} We say that the \emph{local bifurcation condition} is satisfied provided that 
\be \inf_{\lambda > 0} \nu(\lambda) < -1 \label{lbcnu} \tag{LBC}, \ee
or, equivalently, if 
\be \textrm{there exists a nontrivial solution to the linearized problem \eqref{shearsturmliouvilleode} for $n = 1$.} \tag{LBC${}^\prime$} \label{lbcPn} \ee
\end{definition}

These are abstract conditions that are both necessary and sufficient for local bifurcation.  One way to derive an explicit sufficient condition is to require that $\mathscr{R}(\varphi; \epsilon) < -1$, for some $\epsilon > 0$ and a conveniently chosen $\varphi \in \mathscr{A}$.  This is the approach of the next lemma.

\begin{lemma}[Size condition] \label{shearsizeconditionlemma}  Suppose that the prescribed circulation $\Gamma_{\textrm{rel}}$, the pseudo-volumetric mass flux in the air region $p_1$, and the density jump $\jump{\rho}$ collectively satisfy the following size condition:
\be g\jump{\rho}p_1^2 + \int_{p_1}^0 \left( \Gamma_{\mathrm{rel}}(p)^3 + p^2 \Gamma_{\mathrm{rel}}(p) \right)\, dp < 0.\label{shearsizecond} \ee
Then \eqref{lbcnu} holds.
\end{lemma}
\begin{proof}  This can be seen easily by taking 
\[ \varphi(p) := \left\{ \begin{array}{ll} p/p_1, & p_1 \leq p \leq 0 \\
(p-p_1)/(p_1-p_0), & p_0 \leq p \leq p_1. \end{array} \right. \]
and evaluating $\lim_{\lambda \to 0} \mathscr{R}(\varphi; \lambda)$.    Here we have made use of the fact that $a^{(1)} = \Gamma_{\textrm{rel}}$ to express the size condition only in terms of prescribed quantities.  
\qquad \end{proof}

\begin{lemma}[Monotonicity of $\nu$] \label{monotonicitynulemma}  If $\nu(\lambda) < 0$, then $d{\nu}(\lambda)/d\lambda > 0$. \end{lemma}  
\begin{proof}  Let $\mathcal{L}\varphi := -(a^3 \varphi_p)_p$.  Then for each $\lambda$, let $\varphi \in \mathscr{A}$ solve the problem 
\[ \left\{ \begin{array}{ll} \mathcal{L} \varphi = \nu(\lambda) a \varphi, &  \textrm{in } (p_0,p_1) \cup (p_1, 0) \\
\jump{a^3 \varphi_p} = g\jump{\rho} \varphi, & \textrm{on } I, \\
\varphi = 0, & \textrm{on } T \cup B.\end{array} \right. \]
From this we compute
\[ \left\{ \begin{array}{ll} \mathcal{L} \dot{\varphi} = (3a^2 \dot{a} \varphi_p)_p + \dot{\nu} a \varphi + \nu \dot{a} \varphi + \nu a \dot{\varphi}, &  \textrm{in } (p_0,p_1) \cup (p_1, 0) \\
\jump{a^3 \dot{\varphi}_p} + \jump{3a^2 \dot{a} \varphi} = g\jump{\rho} \dot{\varphi}, & \textrm{on } I, \\
\dot{\varphi} = 0, & \textrm{on } T \cup B,\end{array} \right. \]
where a dot denotes differentiation with respect to $\lambda$.   Letting $(\cdot, \cdot)$ be the standard $L^2$-inner product, we have
\[ (\mathcal{L} \dot{\varphi}, \varphi) - ( \mathcal{L} \varphi, \dot{\varphi}) = \dot{\nu} (a \varphi, \varphi) + \nu (\dot{a} \varphi, \varphi) + \left( (3a^2 \dot{a} \varphi_p)_p, \varphi\right).\]
Upon integrating the last term by parts, we find ultimately that
\[  (\mathcal{L} \dot{\varphi}, \varphi) - ( \mathcal{L} \varphi, \dot{\varphi}) = \dot{\nu} (a \varphi, \varphi) + \nu (\dot{a} \varphi, \varphi) -3( a^2 \dot{a} \varphi_p, \varphi_p) - \jump{3a^2 \dot{a} \varphi_p} \varphi(p_1).  \]

However, if we simply compute the difference using integration by parts, we see that
\[ (\mathcal{L} \dot{\varphi}, \varphi) - ( \mathcal{L} \varphi, \dot{\varphi})  = \jump{a^3 \dot{\varphi}_p} \varphi(p_1) - \jump{a^3 \varphi_p} \dot{\varphi}(p_1).\]
Equating the last two lines and exploiting the jump conditions satisfied by $\varphi$ and $\dot{\varphi}$, we obtain the following identity
\[ \dot{\nu} ( a \varphi, \varphi) + \nu (\dot{a} \varphi, \varphi) = 3(a^2 \dot{a} \varphi_p, \varphi_p).  \]
Observe that $a > 0$ and $\dot{a} \geq 0$ (in fact, $\dot{a}^{(2)} = 1$ and $\dot{a}^{(1)} = 0$), hence the right-hand side above is positive.  If $\nu$ is negative, it must then be the case that $\dot{\nu} > 0$.  
 \qquad \end{proof}

\begin{lemma}[Existence of the minimizer] \label{shearexistenceminimizer} Suppose that the size condition \eqref{shearsizecond} holds.  Then there exists a unique value $\lambda^* > 0$ such that $\nu(\lambda^*) = -1$.  Equivalently, there is a unique value of $\lambda$ for which there is a nontrivial solution to \eqref{shearlinearproblem} with the ansatz $m(q,p) = M(p) \cos{q}$.  Moreover, $Q$ is an invertible function of $\lambda$ in a neighborhood of $\lambda^*$.   
\end{lemma}
\begin{proof}  From Lemma \ref{boundRlemma}, we have that $\nu(\lambda) > -1$ for $\lambda$ sufficiently large, while from Lemma \ref{shearsizeconditionlemma} we know that $\nu(\lambda) < -1$ for $\lambda$ sufficiently small.  By continuity, there exists $\lambda^*$ such that $\nu(\lambda^*) = -1$.  Moreover, the monotonicity of $\nu$ established in Lemma \ref{monotonicitynulemma} implies that $\lambda^*$ is unique.  

Next, since $Q$ is a concave function of $\lambda$ according to \eqref{defshearlaminarQ}, it will be a bijection locally so long as $\lambda^*$ does not coincide with the unique critical point $\lambda_0$ of $Q$.  Recall  that $\lambda_0$ satisfies
\[ \lambda_0^3 = -g \jump{\rho}(p_1 - p_0),\]
or, put more suggestively,
\[ \int_{p_0}^{p_1} a_{\lambda_0}^{-3} \, dp = \lambda_0^{-3} (p_1 - p_0) = -\frac{1}{g \jump{\rho}}.\]
Then, for any $\varphi \in \mathscr{A}$, we have the estimate
\begin{align*} \varphi(p_1)^2 = \left(\int_{p_0}^{p_1} \varphi_p \, dp \right)^2 & \leq \left( \int_{p_0}^{p_1} a_{\lambda_0}^3 \varphi_p^2 \, dp \right) \left( \int_{p_0}^{p_1} a_{\lambda_0}^{-3} \, dp \right) \\
& = -\frac{1}{g \jump{\rho}} \int_{p_0}^{p_1} a_{\lambda_0}^3 \varphi_p^2 \, dp.\end{align*}
It follows that 
\[ 0 \leq g\jump{\rho} \varphi(p_1)^2 + \int_{p_0}^{p_1} a_{\lambda_0}^3 \varphi_p^2 \, dp \leq g\jump{\rho} \varphi(p_1)^2 + \int_{p_0}^{0} a_{\lambda_0}^3 \varphi_p^2 \, dp. \]
As this implies $\mathscr{R}(\varphi; \lambda_0) \geq 0$ for all $\varphi \in \mathscr{A}$, we conclude that $\mu(\lambda_0) \neq -1$, and hence $\lambda^* \neq \lambda_0$.  This completes the proof.  
\qquad \end{proof}

\subsection{Proof of local bifurcation}

Let us set $h(q,p) = H(p; \lambda) + m(q,p)$ and $Q = Q(\lambda)$.  The equation satisfied  by $m$ is thus 
\[ \mathcal{F}(\lambda, m) = 0\]
where $\mathcal{F} = (\mathcal{F}_1, \mathcal{F}_2, \mathcal{F}_3, \mathcal{F}_4) : \mathbb{R} \times X \to Y$ is defined by
\be \label{shearF} \begin{split} \mathcal{F}_1(\lambda, w) & := (1+(w_q^{(1)})^2)(w_{pp}^{(1)}+H_{pp}) + w_{qq}^{(1)}(w^{(1)}_p+H_p)^2  \\
 & \qquad - 2w_q^{(1)} (H_p +w_p^{(1)}) w_{pq}^{(1)} + \gamma (H_p + w^{(1)}_p)^3  \\
 \mathcal{F}_2(\lambda, w) & := (1+(w_q^{(2)})^2)(w_{pp}^{(2)}+H_{pp}) + w_{qq}^{(2)}(w^{(2)}_p+H_p)^2  \\
 & \qquad - 2w_q^{(2)} (H_p +w_p^{(2)}) w_{pq}^{(2)} + \gamma (H_p + w^{(2)}_p)^3\\
\mathcal{F}_3(\lambda, w) & := -\jump{\frac{\displaystyle 1+w_q^2}{\displaystyle (H_p+w_p)^2}} - 2g\jump{\rho} (w+H)|_I + Q \\
\mathcal{F}_4(\lambda, w) & := \left(w-d(w)+ H-d(H) - \ell\right)|_T. 
 \end{split} \ee
 Here, the Banach spaces $X$ and $Y = Y_1 \times Y_2 \times Y_3 \times Y_4 $ are
 \[ X := \{ h \in C_{\textrm{per}}^{2+\alpha}(\overline{D} \setminus I) \cap C^{\alpha}_{\textrm{per}}(\overline{D}): h(p_0) = 0, \, h^{(i)} \in C_{\textrm{per}}^{1+\alpha}(\overline{D_i})\},\] 
 \[ Y_1 :=  C_{\textrm{per}}^{2+\alpha}(\overline{D_1} \setminus I) \cap C^{\alpha}_{\textrm{per}}(\overline{D_1}), \qquad Y_2 := C_{\textrm{per}}^{2+\alpha}(\overline{D_2} \setminus I) \cap C_{\textrm{per}}^{\alpha} (\overline{D_2})\]
 \[ Y_3 :=  C_{\textrm{per}}^{\alpha}(I), \qquad Y_4 := C_{\textrm{per}}^{2+\alpha}(T).\]

 For later reference, we now record the Fr\'echet derivative of $\mathcal{F}$ with respect to $w$ at $(\lambda^*,0)$.
 \be \begin{split}  \label{shearFw}
 \mathcal{F}_{iw}(\lambda^*, 0)\varphi & = \left(\partial_p^2 + H_p^2 \partial_q^2 + 3\gamma H_p^2 \partial_p \right) \varphi^{(i)} \qquad \textrm{for } i = 1,2, \\
 \mathcal{F}_{3w}(\lambda^*,0)\varphi & = 2\jump{H_p^{-3} \varphi_p} - 2g \jump{\rho} \varphi, \\
 \mathcal{F}_{4w}(\lambda^*,0)\varphi & = \left(\varphi - d(\varphi) \right)_T.
 \end{split} \ee 

\begin{lemma}[Nullspace]  \label{shearnullspacelemma} The null space of $\mathcal{F}_w(\lambda^*,0)$ is one-dimensional.  
\end{lemma}
\begin{proof}
Let $\varphi$ be an element of the nullspace.  By the evenness built into the definition of $X$, we may expand $\varphi$ via a cosine series:
\[ \varphi(q,p) = \sum_{n=0}^\infty \varphi_n(p) \cos{(nq)}.\]
It follows that, 
\[ \mathcal{F}_w(\lambda^*,0)(\varphi_n(p) \cos{(nq)}) = 0, \qquad n \geq 0.\]
Equivalently, we must have that $\varphi_n$ solves \eqref{shearsturmliouvilleode}.  By Lemma \ref{shearexistenceminimizer} and the definition of $\lambda^*$, we know that $\varphi_1$ is nontrivial.  We have already seen that there are no nontrivial solutions of \eqref{shearsturmliouvilleode} with $n = 0$, and hence $\varphi_0$ must vanish identically.  If $\varphi_n \not\equiv 0$ for some $n > 1$, then it belongs to the admissible set $\mathscr{A}$ and hence $\mathscr{R}(\varphi_n; \lambda^*) = -n^2 < -1$.    This contradicts the definition of $\lambda^*$ as the minimizer.  Thus all the $\varphi_n$ with $n > 1$ vanish identically.  We conclude that the nullspace is generated by $\varphi_1= \varphi^*$.  \qquad
\end{proof}

\begin{lemma}[Range] \label{shearrangelemma} $\mathcal{A} = (\mathcal{A}_1, \mathcal{A}_2, \mathcal{A}_3, \mathcal{A}_4, \mathcal{A}_5) \in Y$ is in the range of $\mathcal{F}_w(\lambda^*,0)$ if and only if it satisfies the following orthogonality condition
 \be \label{shearrangecondition} \int\!\!\!\int_{D_1} a^3 \mathcal{A}_1 \varphi^* \, dq \, dp + \int\!\!\!\int_{D_2} a^3 \mathcal{A}_2 \varphi^* \, dq \, dp + \frac{1}{2} \int_{I} \mathcal{A}_3 \varphi^* \, dq + \int_T a^3 \mathcal{A}_4 \varphi_p^* \, dq = 0. \ee
\end{lemma}
\begin{proof}  First assume that $\mathcal{A}$ is in the range of $\mathcal{F}_w(\lambda^*,0)$.  Then there exists $\varphi \in X$ such that $\mathcal{F}_w(\lambda^*,0) \varphi = \mathcal{A}$.  Writing $\mathcal{F}_{iw}(\lambda^*,0)$ for $i = 1,2$ in self-adjoint form, we have
\[ a^{-3} \partial_p  \left(a^3 \partial_p \varphi^{(i)}\right) + a^{-2} \partial_q^2 \varphi^{(i)} = \mathcal{A}^{(i)}, \qquad i = 1,2.\]
Thus
\begin{align*} \left( a^3 \varphi^*, \mathcal{A}_1 \right)_{L^2(D_1)} + \left(a^3 \varphi^*, \mathcal{A}_2\right)_{L^2(D_2)} & = \int\!\!\!\int_{D_1}  \left( (a^3\varphi_{p})_p + a \varphi_{qq} \right) \varphi^* \, dq \, dp \\
& \qquad + \int\!\!\!\int_{D_2}  \left( (a^3\varphi_{p})_p + a \varphi_{qq} \right)  \varphi^* \, dq \, dp \\
& = - \int\!\!\!\int_{D_1}  a^3 \varphi_{p}  \varphi_p^* \, dq \, dp \\
& \qquad - \int\!\!\!\int_{D_2}  a^3 \varphi_{p}  \varphi_p^* \, dq \, dp  - \int_I \jump{a^3 \phi_p \varphi^*} \, dq \\
& \qquad\qquad  + \int\!\!\!\int_{D_1 \cup D_2}  a \varphi  \varphi_{qq}^* \, dq \, dp. \end{align*}
Here we have exploited the periodicity and the fact that $\varphi^*$ vanishes identically on $T$.   This is precisely what we have in \eqref{idealrangecomputation}, and, since the conditions on the interface are unchanged,  proceeding as in that proof shows that \eqref{shearrangecondition} is necessary.  

The sufficiency of \eqref{shearrangecondition} follows from a straightforward generalization of the argument given in Lemma \ref{idealrangelemma} which we omit.  \qquad
\end{proof}

\begin{lemma}[Transversality] \label{sheartransverselemma} The following transversality condition holds:
\be \label{idealtranscond} \mathcal{F}_{\lambda w}(\lambda^*,0) \varphi^* \not\in \mathcal{R}(\mathcal{F}_w(\lambda^*,0)). \ee
Here $\varphi^*$ denotes a generator of $\mathcal{N}(\mathcal{F}_w(\lambda^*,0))$.  
\end{lemma}
\begin{proof} In view of Lemma \ref{idealrangelemma}, it suffices to show that $\mathcal{A} := \mathcal{F}_{\lambda w}(\lambda^*,0) \varphi^*$ fails to satisfy the orthogonality condition \eqref{shearrangecondition}.  That is, if we put 
\[ \Xi := \int\!\!\!\int_{D_1} a^3 \mathcal{A}_1 \varphi^* \, dq \, dp + \int\!\!\!\int_{D_2} a^3 \mathcal{A}_2 \varphi^* \, dq \, dp + \frac{1}{2} \int_{I} \mathcal{A}_3 \varphi^* \, dq + \int_T a^3 \mathcal{A}_4 \varphi_p^* \, dq,\]
we must prove $\Xi \neq 0$. To do this, we first compute
\begin{align*}
\mathcal{F}_{1\lambda w}(\lambda^*,0) \varphi^* & = 0, \\
 \mathcal{F}_{2\lambda w}(\lambda^*,0) \varphi^* & = -\frac{2}{(\lambda^*)^3}  \varphi_{qq}^* -\frac{6\gamma}{(\lambda^*)^3} \varphi_p^*, \\
 \mathcal{F}_{3\lambda w}(\lambda^*,0) \varphi^* & =  \left(-3(\lambda^*)^2 (\varphi_p^*)^{(2)} \right)_I, \\
 \mathcal{F}_{4\lambda w}(\lambda^*,0) \varphi^* & = 0.
\end{align*}
By the equation satisfied by $\varphi^*$ \eqref{shearsturmliouvilleode}, we see that 
\[ -\partial_q^2 \varphi^* = \varphi^* = a^{-1} \partial_p  (a^3 \partial_p \varphi^*).\]
Thus, 
\[ (\lambda^*)^2  ( \varphi^* \varphi_p^*)_p = (\lambda^*)^2 (\varphi_p^*)^2 + (\varphi^*)^2, \qquad \textrm{in } D_2.\]
From this we deduce,
\be \begin{split} \int\!\!\!\int_{D_2} a^3 \mathcal{A}_2 \varphi^* \, dq \, dp &= 2 \int\!\!\!\int_{D_2} (\varphi^*)^2  \, dq\, dp -6 \gamma \int\!\!\!\int_{D_2} \varphi^* \varphi_p^* \, dq \, dp \\
& = 2 \int\!\!\!\int_{D_2} (\varphi^*)^2  \, dq\, dp, \end{split} \label{sheartrans1} \ee
and 
\be \begin{split} \frac{1}{2} \int_I \mathcal{A}_3 \varphi^* \, dq  &= -\frac{3(\lambda^*)^2}{2} \int_I (\varphi_p^*)^{(2)} \varphi^* \, dq \\
& = -\frac{3 (\lambda^*)^2}{2} \int\!\!\!\int_{D_2} (\varphi_p^*)^2 \, dq\, dp - \frac{3}{2} \int\!\!\!\int_{D_2} (\varphi^*)^2 \,dq \, dp.
\end{split} \label{sheartrans2} \ee
In light of \eqref{sheartrans1}--\eqref{sheartrans2}, and the calculated values of the remaining components of $\mathcal{A}$, 
\[ \Xi = \frac{1}{2} \int\!\!\!\int_{D_2} (\varphi^*)^2 \, dq \, dp - \frac{3}{2} (\lambda^*)^2 \int\!\!\!\int_{D_2} (\varphi_p^*)^2 \, dq\, dp. \]
Since the equation satisfied by $(\varphi^*)^{(2)}$ is the same as for the irrotational atmosphere case treated in the previous section, we know that $\Xi < 0$ by the same argument as Lemma \ref{idealtransverselemma} (cf. \eqref{idealtrans3}).  This completes the proof of the lemma. \qquad
\end{proof}

\pfthm{\ref{shearlocalbifurcationtheorem}}  Under the assumption that the compatibility condition \eqref{shearcompatibilitycond} and the local bifurcation condition \eqref{lbcnu} hold, Lemmas \ref{shearnullspacelemma}--\ref{sheartransverselemma} verify that the hypotheses for Theorem \ref{crandallrabinowitz} are satisfied.  The conclusions of Theorem \ref{shearlocalbifurcationtheorem} follow immediately.  \qquad \endproof

\section{Local bifurcation with an unbounded irrotational atmosphere} \label{unboundedidealsection}
 In this section and the next we consider the situation where the atmosphere region $\Omega^{(1)}$ has infinite vertical extent.     We begin by supposing that the flow in both the water and air are irrotational.   In the moving frame, the governing equations for the stream function formulation are as in \eqref{psieq} (taking $\rho^{(i)}$ to be constants, $d = 1$, and $\gamma \equiv 0$), with the only difference manifesting in the boundary condition at infinity:
\be \left\{ \begin{array}{ll} \Delta \psi  = 0, & \textrm{in } \Omega, \\
\displaystyle \jump{|\nabla \psi|^2} + 2g \jump{\rho} \left(\eta+1\right)   - Q = 0, & \textrm{on } y = \eta(x), \\
\psi = 0, & \textrm{on } y = \eta(x), \\
\psi = -p_0, &  \textrm{on } y = -1, \\
\nabla^\perp \psi \to(-\lambda, 0) \textrm{ uniformly in $x$} &  \textrm{as } y \to \infty.\end{array} \right. \label{unboundedidealpsieq} \ee
Here $\lambda = \sqrt{\rho^{(1)}}(U-c)$ is the speed of the undisturbed wind in the co-moving frame multiplied by the square root o the density of the air.  

The change in the domain necessitates a change in the spaces in which we seek solutions.  In particular, we must specify the behavior as $y \to \infty$.   As can be seen in \eqref{unboundedidealpsieq}, for the irrotational regime we are interested the case where $(u,v) \to (U,0)$, for some constant $U < c$ as $y \to \infty$, which is equivalent to requiring that $\nabla^\perp \psi$ limits to $(-\lambda, 0)$, for some positive constant $\lambda$.   With that in mind, we define 
\begin{align*}
\mathscr{S}_0 & := \{ (u,v,\varrho, \eta) \in \mathscr{S} : \exists U < c,  (u,v) \to (U, 0) \textrm{ uniformly as } y \to \infty \} \\
\mathscr{S}_0^\prime & := \{ (Q, \psi, \eta) \in \mathscr{S}^\prime : \exists \lambda > 0, \nabla^\perp \psi \to (-\lambda, 0) \textrm{ uniformly as } y \to \infty \}.  \end{align*}

\begin{theorem} [Local bifurcation with unbounded irrotational atmosphere]  \label{unboundedideallocalbifurcationtheorem} Suppose that the volumetric mass flux in the water region $p_0$ and the density jump $\jump{\rho}$ satisfy the local bifurcation condition 
\be p_0^2 \coth{1} + g\jump{\rho} > 0, \label{unboundedideallbc} \ee
then the following statements are true.  
\begin{romannum}
\item  There exists a continuous curve of non-laminar solutions to the stream function equation for irrotational flow in the air and water, and unbounded atmosphere region  \eqref{unboundedidealpsieq}
\[ \mathcal{C}_{\mathrm{loc}}^{\prime} = \{ (Q(s),\psi(s), \eta(s)) \in \mathscr{S}_0^\prime : |s| < \epsilon \}, \]
for $\epsilon > 0$ sufficiently small, such that $(Q(0), \psi(0)) = (Q^*, \Psi_0^*)$, and, in a sufficiently small neighborhood of $(Q^*, \Psi_0^*)$ in $\mathbb{R}\times X$, $\mathcal{C}_{\mathrm{loc}}^\prime$ comprises all non-laminar solutions.  In particular, we have
\[ \eta(s)(\cdot) = s \cos{(\cdot)} + o(s) \qquad \textrm{in } C_{\mathrm{per}}^{2+\alpha}(\mathbb{R}).\] 

\item  There is a corresponding continuous curve 
\[ \mathcal{C}_{\mathrm{loc}} = \{ (Q(s), u(s), v(s), \varrho(s), P(s), \eta(s)) \in \mathscr{S}_0 : |s| < \epsilon \}\]
of small amplitude solutions to the Eulerian problem which likewise captures all non-laminar solutions in a sufficiently small neighborhood of the point of bifurcation.  
\end{romannum}
\end{theorem}

 Unlike to the explicit size condition \eqref{shearsizecond} obtained in the rotational lidded regime, \eqref{unboundedideallbc} is both necessary and sufficient.  We caution, however, that \eqref{unboundedideallbc} is stated in dimensional variables, and so some additional work must be done before attempting to draw physical conclusions.   For instance, note that we are taking $d = 1$, effectively setting the length scale of the system to match the depth of the ocean (which is typically greater than $3000$ meters), and constructing minimally $2\pi$-periodic waves.   Clearly such waves are not typical.  

A simple generalization of Theorem \ref{unboundedideallocalbifurcationtheorem} can be made where we allow an arbitrary depth $d > 0$, and perturbations of the form $2\pi/k$, $k \geq 1$.     Let $k^*$ be the smallest natural number such that 
\be p_0^2 k^* \coth{(k^*d)} + g d^2 \jump{\rho} > 0.\label{unboundedidealgenerallbc} \ee
Then the conclusions of Theorem \ref{unboundedideallocalbifurcationtheorem} hold, with the only modification being that the non-laminar solutions are of the form 
\[ \eta(s)(x) = s \cos{(k ^*x)} + o(s).\]
and the depth of the water region is $d$.  Notice that $k^* = 1$ and $d = 1$ if and only if \eqref{unboundedideallbc} holds.

\subsection{The transformed problem} \label{unboundedidealtranssection}

At a mathematical level, the major difference in going from the
lidded and to the unbounded case is the loss of compactness in
the domain.  If we were to attempt to use semi-Lagrangian
coordinates and the height equation formulation, we would not
expect that the operator $\mathcal{F}$ to be Fredholm.  This is a
substantial technical obstacle, but one that has been grappled
with extensively in the literature of traveling waves in oceans
of infinite depth (cf., e.g.,
\cite{hur2006global,hur2008solitary}) and solitary waves (cf.,
e.g., \cite{beale1977solitary}).  Typically, the strategy is to
rely on concentration compactness arguments or a Nash--Moser
iteration scheme.  But aside from the inherent complexities of
these tools, were we to attempt this approach, we would still be
forced to require the absence of stagnation points and critical
layers in order to justify the Dubreil-Jacotin transformation.
This an especially restrictive assumption in the unbounded
atmosphere regime.  To see why, consider flows with a constant
nonzero vorticity in the air region, which is the subject of the
next section.  A laminar flow $(U, 0)$ of this type will
necessarily satisfy $|U| \to \infty$ as $y \to \infty$.  Thus,
depending on the sign of the vorticity,
critical layers are \emph{expected}, though they are  neutered (see the remark following
Lemma \ref{unboundedshearlaminarlemma}).

Rather than adopting semi-Lagrangian coordinates, therefore, we
shall employ a more robust (though less elegant) transformation
to fix the domain.  Let $\overline{\Omega_0} =
\overline{\Omega_0^{(1)}} \cup \overline{\Omega_0^{(2)}}$, where
\[ \overline{\Omega_0^{(1)}} =  [-\pi, \pi] \times [0, \infty), \qquad \overline{\Omega_0^{(2)}} =  [-\pi, \pi]  \times [-1, 0].\]
Suppose $\eta(\epsilon, \cdot)$ is a one-parameter family of free surfaces with $\eta(0, \cdot) \equiv 0$.  Then we may take $\Omega(\epsilon)$ and $\psi(\epsilon, \cdot)$ to be the corresponding families of fluid domains and stream functions, respectively.  Let $T = T(\epsilon, x, y) : \Omega(\epsilon) \to \Omega_0$ be a smooth diffeomorphism mapping $\Omega(\epsilon)$ to  $\Omega_0$ for each $\epsilon \geq 0$ such that 
\begin{align*} T(0, \cdot) &= \iota_{\Omega_0} \textrm{ (the identity map on $\Omega_0$)} \\
 T(\epsilon, x, \eta(x)) &= (x, 0) \\
  T(\epsilon, x, -1) &= -1\\
   [T(\epsilon, \cdot) - \iota_{\Omega(\epsilon)}] \to 0 \textrm{ as } y \to \infty, & \textrm{ uniformly in $x$ and $\epsilon$}. \end{align*}   
Any such map will serve the purpose of fixing the domain, but for simplicity we make a particular choice.  Suppose that in a neighborhood of the $x$-axis, $T$ is just the flattening map:
\be \begin{split} T_1(\epsilon, x,y) &:= \pi_1(x,y) := x, \\
 T_2(\epsilon, x, y)& := \frac{y-\eta(\epsilon, x)}{1+\eta(\epsilon, x)} \chi(\epsilon, y) + y(1-\chi(\epsilon, y)), \end{split}\label{unboundedidealdefT} \ee
 where $\chi$ is a fixed, smooth cutoff function with support on  $\{ (x,y) : y < \eta(\epsilon, x) + 2 \}$ and chosen so that  $T(\epsilon, \cdot)$ is a diffeomorphism.  Because we are fixing a representation for $T$, the only unknowns in the problem are $\psi$, $\eta$, and $Q$.  In particular, $T$ is determined entirely by $\eta$.  We remark that this map has been the basis for a number of studies of water waves.  To reference only the most immediately relevant results, we point out that Wahl\'en, Ehrstr\"om et al. rely on it in  their investigations of steady waves with critical layers (cf. \cite{wahlen2009critical,ehrnstrom2010multiple,ehrnstrom2010interior}).

For notational convenience, we denote the inverse of $T(\epsilon, \cdot)$ by $S(\epsilon, \cdot)$.  We shall also use the convention that the coordinates in $\Omega_0$ are in the variables $(\bar{x}, \bar{y})$, while the unbarred variables indicates coordinates in $\Omega(\epsilon)$, $\epsilon \neq 0$.  

Define the transformed stream function $\Psi$ by the relation
\[ \psi(\epsilon, \cdot) = [\Psi(\epsilon, \cdot)] \circ T(\epsilon, \cdot),\]
or, equivalently, 
\be \Psi(\epsilon, \cdot) := [\psi(\epsilon, \cdot)] \circ S(\epsilon, \cdot).\label{unboundeddefPsi} \ee
We shall suppose that $\Psi(0,\bar{x}, \bar{y}) = \Psi_0(\bar{y})$, meaning that the unperturbed flow is laminar.  

Let $\partial_i$ denote partial differentiation with respect to the $i$-th physical variable, for $i = 1,2$, and, for any function $f$ of two variables, let $f_{,i} := \partial_i f$.  Then an elementary computation confirms that, $-\Delta \psi(\epsilon,\cdot) = \gamma(\psi)$ in $\Omega(\epsilon)$ if and only if $\Psi(\epsilon, \cdot)$ satisfies the following equation in $\Omega_0$:
\be \mathcal{E}(\eta) \Psi := A_{ij} \partial_i \partial_j \Psi + B_i \partial_i \Psi + \gamma(\Psi) = 0 \qquad \textrm{in } \Omega_0, \, \textrm{for all $\epsilon$.} \label{unboundedPsisemilinear} \ee
Here we are adopting the summation convention over repeated indices, $\gamma$ represents the vorticity strength function for the flow ($\gamma = 0$ for irrotational flow), and 
\be \begin{split} A_{ij} = A_{ij}(\eta) & := [ (\partial_k T_i)(\epsilon, \cdot) (\partial_k T_j )(\epsilon, \cdot)] \circ S(\epsilon, \cdot) \\
 B_i = B_i(\eta) &:= [(\partial_j \partial_j T_i)(\epsilon, \cdot)] \circ S(\epsilon, \cdot). \label{unboundeddefAijBi} \end{split} \ee
Note that are we stating that $\mathcal{E}$ depends on $\eta$, while only $T$ occurs above.  This is valid because $T$ is determined uniquely by $\eta$ in view of \eqref{unboundedidealdefT}.

The jump condition on the boundary in \eqref{unboundedidealpsieq} can likewise be reformulated in terms of $\Psi$, resulting in the following
\be \jump{(C_{ij} \partial_i \Psi)^2} + 2g\jump{\rho} H - Q = 0 \qquad \textrm{on } \bar{y} = 0,\label{unboundedPsijump} \ee
where $H(\epsilon, \bar{x}, \bar{y}) := S_2(\epsilon, \bar{x}, \bar{y}) - S_2(\epsilon, \bar{x}, -1)$ is the height above the ocean bed in the terms of the coordinates $(\bar{x}, \bar{y})$, and 
\be C_{ij} = C_{ij}(\eta) := [(\partial_j T_i)(\epsilon, \cdot)] \circ S(\epsilon, \cdot). \label{unboundeddefCij} \ee  
The Dirichlet conditions for $\Psi$ derive from those for $\psi$ and the definition of $T$: 
\be \left\{\begin{array}{ll} \Psi = -p_0 & \qquad \textrm{on } \bar{y} = -1, \\
\Psi = 0  &\qquad \textrm{on } \bar{y} = 0.
\end{array} \right.\label{unboundedPsidirichlet} \ee

Lastly, the Neumann boundary condition at $y = +\infty$ translates to the same condition for $\Psi$, since we have that $T$ asymptotically approaches the identity.   
\be \nabla^\perp \Psi \to (-\lambda, 0) \textrm{ as $\bar{y} \to \infty$ uniformly in $\bar{x}$ and $\epsilon$}. \label{unboundedPsicirculationcond} \ee

In summary, we find that there exists a non-laminar solution to $(\psi, \eta, Q)$ to \eqref{unboundedidealpsieq} provided that there is exists $(\eta, \lambda, Q)$ for which there are nontrivial solutions $(\Psi, H, Q(\lambda))	$ to \eqref{unboundedPsisemilinear}--\eqref{unboundedPsicirculationcond}.  To make the latter problem tractable, we suppose that there is a unique solution $\Psi$ for given $(\eta, \lambda)$.  This is valid since, when $\eta = 0$, $T = \iota_{\Omega_0}$ and thus $\mathcal{E}(\lambda,0) = \Delta$.  It follows that  the operator is an isomorphism for $\epsilon$ in a neighborhood of $0$.   Explicitly, we \emph{define} $\Psi^{(1)} = \Psi^{(1)}(\lambda,\eta)$ to be the solution of 
\be  \left\{ \begin{array}{ll} \mathcal{E}(\lambda,\eta) \Psi^{(1)} = 0 & \textrm{in } \Omega_0^{(1)}, \\ \Psi^{(1)} = 0 & \textrm{on } \bar{y} = 0, \\
 \nabla^\perp \Psi^{(1)} \to (-\lambda, 0) \textrm{ as } y \to \infty & \textrm{uniformly in $\bar{x}$}. \end{array} \right.\label{unboundedidealdefPsi1} \ee
 and define $\Psi^{(2)} = \Psi^{(2)}(\lambda,\eta)$ to be the unique solution of
 \be  \left\{ \begin{array}{ll} \mathcal{E}( \lambda,\eta) \Psi^{(2)} = 0 & \textrm{in } \Omega_0^{(2)}, \\ \Psi^{(2)} = 0 & \textrm{on } \bar{y} = 0, \\
\Psi^{(2)} = -p_0 & \textrm{on } \bar{y} = -1. \end{array} \right.\label{unboundedidealdefPsi2} \ee
Then solving \eqref{unboundedPsisemilinear}--\eqref{unboundedPsicirculationcond} is equivalent to the following: find $(\lambda,\eta)$ such that $\mathcal{G}(\lambda,\eta) = 0$, where
\be \mathcal{G}(\lambda,\eta) := \jump{(C_{ij}( \eta,\lambda) \partial_i \Psi(\eta,\lambda))^2} + 2g\jump{\rho} (H(\eta))|_{\bar{y} = 0} - Q.\label{unboundedGequation} \ee
Notice that a laminar flow corresponds to a solution where $\eta = 0$, and $H = \pi_2+1$  (since we have dictated that $T_1 = \pi_1$, this is equivalent to saying $T = \iota$).

\subsection{Linearization} \label{unboundedideallinearsection}
Echoing the approach in \S\ref{liddedidealsection} and \S\ref{liddedvorticalsection}, we begin by proving the existence of a one-parameter family of laminar flows.  We the proceed to linearize \eqref{unboundedGequation} along this family, in order to lay the groundwork for a bifurcation theory argument.

\begin{lemma}[Laminar flows] \label{unboundedideallaminarlemma} There exists a one-parameter family of laminar solutions $(Q(\lambda), 0)$ to \eqref{unboundedGequation} for $\lambda \geq 0$.  The corresponding family of laminar transformed streamed function $\Psi_0(\lambda)$ are given by
\be \Psi_0(\lambda, \bar{y}) = \left\{ \begin{array}{ll} -\lambda \bar{y} & \textrm{for } \bar{y} > 0, \\
p_0 \bar{y} & \textrm{for } y \leq 0, \end{array} \right. \label{unboundedidealPsi0} \ee
and
\be Q(\lambda) := 2(\lambda^2 - p_0^2) + 2g\jump{\rho}.  \label{unboundedidealQ} \ee
\end{lemma}
\begin{proof}
For the flow to be laminar, we must have that $H = \pi_2+1$, so that, from \eqref{unboundedidealpsieq} we see that the transformed laminar stream function $\Psi_0 = \Psi_0(\lambda, \pi_2)$ is harmonic in $\Omega_0^{(1)} \cup \Omega_0^{(2)}$ and has the following boundary data:
\[ \Psi_0(0) = 0, \qquad \Psi_0(-1) = -p_0, \qquad \Psi_0^\prime \to -\lambda \textrm{ as } y \to \infty. \]
We deduce that it must be as in \eqref{unboundedidealPsi0}.  The formula for $Q(\lambda)$ then follows from \eqref{unboundedGequation}. \qquad \end{proof}

Fixing $\lambda$, we now compute the Fr\'echet derivatives of $\mathcal{E}$ and $\mathcal{G}$.  In what follows, $D_\eta$ denotes Fr\'echet differentiation with respect to $\eta$, while $\partial_\epsilon$ is the (finite-dimensional) partial derivative with respect to $\epsilon$.  Also, where there is no risk of confusion, we shall suppress the $\lambda$ dependence.  

Let the variations be denoted by
\[ h := (\partial_\epsilon H)|_{\epsilon = 0}, \qquad \Phi :=  (\partial_\epsilon \Psi)|_{\epsilon = 0}, \qquad \phi := (\partial_\epsilon \psi)|_{\epsilon = 0},  \]
\[ \zeta := (\partial_\epsilon \eta)|_{\epsilon = 0}, \qquad \tau := (\partial_\epsilon T)|_{\epsilon = 0}, \qquad \sigma := (\partial_\epsilon S)|_{\epsilon = 0}.\]
\begin{remark} There are a few points that should be made here:
\begin{romannum}
\item  $\phi$ is not continuous over the $\bar{x}$-axis (indeed, it is not well-defined there at all).  
\item An elementary calculus results states $\tau = -\sigma$.  
\item Chasing the definitions, it is obvious that $\sigma_2 = h$.  
\item Since $\Psi$ vanishes on the $\bar{x}$-axis for all values of $\epsilon$, we have 
\begin{align*} 0 & = \Psi(\epsilon, x, T_2(\epsilon, x, \eta(\epsilon, x))) \\
&= \Phi(x, 0) + \Psi_0^\prime(0) \left( \tau_2 + T_{2,2}(0, x, 0) \zeta(x) \right) \end{align*}
By (ii) and the fact that $T_{2}(0, \cdot) = \pi_2(\cdot)$, this implies
\[ h = \zeta \qquad \textrm{on } \bar{y} = 0.\]
In other words, $\zeta$ is the trace of $h$ on the $\bar{x}$-axis.  \end{romannum}
 \end{remark}  

Now 
\begin{align*} \partial_\epsilon A_{ij} &= [T_{i,k\epsilon}(\epsilon, \cdot) T_{j,k}(\epsilon, \cdot) + T_{i,k} (\epsilon, \cdot) T_{j,k\epsilon}(\epsilon, \cdot) ] \circ S(\epsilon, \cdot) \\
& \qquad + S_{\ell, \epsilon}(\epsilon, \cdot)  [ T_{i,k\ell}(\epsilon, \cdot) T_{j,k} (\epsilon, \cdot) + T_{j, k\ell} (\epsilon, \cdot) T_{i, k}(\epsilon, \cdot)] \circ S(\epsilon, \cdot), \end{align*}
and thus, evaluating at $\epsilon = 0$ we find 
\[ (D_\eta A_{ij})(0) = \delta_{jk} \tau _{i,k} + \delta_{ik} \tau_{j,k} = \partial_i \tau_j + \partial_j \tau_i.\]
Here we have used the fact that $S(0, \cdot) = T(0,\cdot) = \iota_{\Omega_0}$.   

The calculation of the linearization of the first-order coefficients proceeds in the same fashion:
\begin{align*}
\partial_\epsilon B_{i} & = [T_{i,jj\epsilon}(\epsilon, \cdot)] \circ S(\epsilon, \cdot) + S_{k, \epsilon}(\epsilon, \cdot) [T_{i,jjk}(\epsilon, \cdot)] \circ S(\epsilon, \cdot) \\
D_\eta B_i(0)  &= \tau_{i,jj} + \sigma_{k} [ \partial_k^2 \delta_{ik} ] \\
& = \partial_j \partial_j \tau_i. \end{align*}
Collecting these two facts, we see that
\begin{align} \left\langle \mathcal{E}_{\eta} (\lambda,0), \Phi \right\rangle &= A_{ij}(0) \partial_i \partial_j \Phi + B_i(0) \partial_i \Phi \nonumber \\
& \qquad +  \left \langle (D_\eta A_{ij})(0), \partial_i \partial_j \Psi_0 \right\rangle  + \left \langle (D_\eta B_i)(0), \partial_i \Psi_0 \right \rangle \nonumber \\
& = \Delta \Phi - \Psi_0^\prime \Delta h. \label{unboundedlinearizedop} \end{align}
Note that the last line follows from observing that $\tau_2 = -\sigma_2 = -h$.   

\begin{remark} This formula can also be obtained formally by noting that 
\begin{align*} \partial_{\epsilon} \left[ \Psi(\epsilon, \cdot )\right]   & = \partial_\epsilon \left[ \psi(\epsilon, S(\epsilon, \cdot)) \right] \\
& = \psi_{,\epsilon}(\epsilon, S(\epsilon, \cdot)) + \psi_{,1}(\epsilon, S(\epsilon, \cdot)) S_{1,\epsilon}(\epsilon, \cdot) + \psi_{,2}(\epsilon, S(\epsilon, \cdot)) S_{2,\epsilon}(\epsilon, \cdot), \end{align*}
and thus, evaluating at $\epsilon = 0$, we have
\[ \Phi = \phi+ \Psi_0^{\prime} h.\]
Taking the Laplacian of both sides of the equation leads to \eqref{unboundedlinearizedop}.  \end{remark}

Next we consider the linearization of $\mathcal{G}$, the transmission boundary condition.  
\[ \partial_\epsilon C_{ij}  = [T_{i,j\epsilon}(\epsilon, \cdot)] \circ S(\epsilon, \cdot) + S_{k, \epsilon}(\epsilon, \cdot)[T_{i,jk}(\epsilon, \cdot)] \circ S(\epsilon, \cdot), \]
whence
\[ (D_\eta C_{ij})(0) = \tau_{i,j} + \sigma_k \partial_k (\delta_{ij}) = - \partial_j \sigma_{i}.\]
Using this identity with \eqref{unboundedGequation}, we compute
\begin{align}
\left\langle \mathcal{G}_\eta( 0,\lambda), \zeta \right\rangle & = 2\jump{ ( \langle (D_\eta C_{ij})(0), h \rangle \partial_i \Psi_0 + C_{ij}(0) \partial_i \Phi) ( C_{ij}(0) \partial_i \Psi_0) } \nonumber \\
& \qquad + 2g\jump{\rho} \zeta \nonumber \\
& = 2 \jump{ (-\Psi_0^\prime \partial_2 h + \partial_2 \Phi)\Psi_0^\prime} + 2g\jump{\rho} \zeta. \label{unboundedlinearizedGcalc1} \end{align} 

As a consequence of our choice of $T$ in \eqref{unboundedidealdefT}, we have 
\[ S_2(\epsilon, \bar{x}, \bar{y}) = (\eta(\epsilon, \bar{x}) + 1)\bar{y} + \eta(\epsilon, \bar{x}) \qquad \textrm{in a neighborhood of $\{\bar{y} = 0\}$,}\]
and thus 
\[ h(\bar{x}, \bar{y}) = \zeta(\bar{x}) \bar{y} + \zeta(\bar{x}) \qquad \textrm{in a neighborhood of $\{\bar{y} = 0\}$}.\]
This implies $\partial_2 h = \zeta$ on $\bar{y} = 0$.  Therefore \eqref{unboundedlinearizedGcalc1} can be written
\[ \left\langle \mathcal{G}_\eta( \lambda,0), \zeta \right\rangle = 2\left(g\jump{\rho} - \jump{(\Psi_0^\prime)^2}\right) \zeta + 2\jump{\Psi_0^\prime \partial_2 \Phi}. \] 
  
Finally, to make sense of this, we compute explicitly the dependence of $\partial_2 \Phi$ on $\zeta$. 
From \eqref{unboundedlinearizedop} we have that $\Phi$ solves
\be  \left\{ \begin{array}{ll} \Delta \Phi = \Psi_0^\prime \Delta h & \textrm{in } \Omega_0 \setminus \{ \bar{y} = 0\}, \\ \Phi= 0 & \textrm{on } \bar{y} = 0, \\
\Phi = 0 & \textrm{on } \bar{y} = -1, \\
 \nabla^\perp \Phi^{(1)} \to (0,0) \textrm{ as } \bar{y} \to \infty & \textrm{uniformly in $\bar{x}$ and $\epsilon$}. \end{array} \right.\label{unboundedidealdefPhi1} \ee
Equivalently: 
\be \left\{ \begin{array}{ll} \Delta (\Psi_0^\prime \Phi-(\Psi_0^\prime)^2 h)  = 0 & \textrm{in } \Omega_0 \setminus \{ \bar{y} = 0\}, \\ (\Psi_0^\prime\Phi-(\Psi_0^\prime)^2 h)^{(1)}= - \zeta ((\Psi_0^\prime)^2)^{(1)}  & \textrm{on } \bar{y} = 0, \\
(\Psi_0^\prime\Phi-(\Psi_0^\prime)^2 h)^{(2)}= - \zeta ((\Psi_0^\prime)^2)^{(2)}  & \textrm{on } \bar{y} = 0, \\
\Psi_0^\prime\Phi-(\Psi_0^\prime)^2 h = 0 & \textrm{on } \bar{y} = -1, \\
 \nabla^\perp (\Psi_0^\prime\Phi - (\Psi_0^\prime)^2 h) \to (0,0) \textrm{ as } \bar{y} \to \infty & \textrm{uniformly in $\bar{x}$ and $\epsilon$}. \end{array} \right. \label{unboundedidealUpsilon} \ee
 Taking $\Upsilon := \Psi_0^\prime \Phi - (\Psi_0^\prime)^2 h$, \eqref{unboundedidealUpsilon} makes clear that $\Upsilon$ depends only on the trace of $h$ on $\{ \bar{y} = 0\}$, i.e., on $\zeta$.  Moreover, because of the evenness and periodicity, we can compute $\Upsilon = \Upsilon(\zeta)$ quite explicitly. 
 Letting $\hat{\cdot}$ designate the Fourier transform in the $\bar{x}$-coordinate with Fourier variable $k$, we have
 \be \begin{split} \widehat{(\Upsilon^{(1)})} (k, y) &= -\lambda^2 \hat{\zeta}(k) \left( \cosh{(ky)} -\sinh{(ky)} \right) \\
   \widehat{(\Upsilon^{(2)})} (k, y) &=  -p_0^2 \hat{\zeta}(k) \left( \coth{(k)} \sinh{(ky)} + \cosh{(ky)} \right). \end{split} \ee
   From this we readily obtain 
   \be \widehat{\jump{\partial_2 \Upsilon}}(k) = k\hat{\zeta}(k)\left( p_0^2 \coth{k} - \lambda^2\right) = (p_0^2 D \coth{D} - \lambda^2 D) \zeta. \label{unboundedideald2Upsilonsymbol} \ee
 That is, we may view $\mathcal{G}_\eta(\lambda,0)$ as a Fourier multiplier:
  \be \begin{split} \left\langle \mathcal{G}_\eta( \lambda,0), \zeta \right\rangle &= 2 g\jump{\rho} \zeta + 2 \jump{\partial_2 \Upsilon(\zeta)} \\
  & = 2(p_0^2 D \coth{D} - \lambda^2 D + g\jump{\rho}) \zeta.\end{split} \label{unboundedidealGmultiplier} \ee
  Here (abusing notation slightly) $D = -i \partial_{\bar{x}}$.  
  
  \subsection{Proof of local bifurcation} With the expressions for $\mathcal{G}_\eta(\lambda, 0)$ derived in the previous subsection, the proof of the local bifurcation theorem is relatively simple.
  
  Define 
  \[ X := C_{\textrm{per}}^{2+\alpha}([-\pi, \pi]), \qquad Y := C_{\textrm{per}}^{1+\alpha}([-\pi,\pi]).\]
  In what follows, $\mathcal{G}$ is considered as an operator with domain $\mathbb{R} \times X$ and codomain $Y$.  
  
  \begin{lemma}[Null space and range]  \label{unboundedidealnullspacelemma}  Under the assumption that the local bifurcation condition \eqref{unboundedideallbc} holds, there exists a $\lambda^* \geq 0$ such that the following statements are true.
  \begin{romannum}
  \item $\mathcal{G}_\eta(\lambda^*, 0)$ is a Fredholm operator of index 0; 
  \item   $\mathcal{N}(\mathcal{G}_\eta(\lambda^*, 0))$ is one-dimensional and spanned by $\zeta^*(\bar{x}) :=  \cos{\bar{x}}$; moreover
  \item  $\xi \in \mathcal{R}(\mathcal{G}_\eta(\lambda^*, 0))$ if and only if $\hat{\xi}(1) = 0$.  \end{romannum}\end{lemma}
  \begin{proof}  By \eqref{unboundedidealGmultiplier}, we know that $\zeta \in \mathcal{N}(\mathcal{G}_\eta(\lambda^*, 0))$ if and only if it satisfies
  \[ (p_0^2 k \coth{k} - \lambda^2 k + g \jump{\rho}) \hat{\zeta}(k) = 0, \qquad \textrm{for all } k \geq 0.\]
  For each $\lambda \geq 0$,  
  \[ m( k;\lambda):=  p_0^2 k \coth{k} - \lambda^2 k +g\jump{\rho} \]
  is injective as a function of $k$, and, fixing $k \geq 0$, $m(k; \lambda) \to -\infty$ as $\lambda \to \infty$.  The local bifurcation condition implies that $m(1;\lambda) > 0$, and hence there exists a $\lambda^*$ such that 
  \be p_0^2 k \coth{k} - (\lambda^*)^2 k + g\jump{\rho} = 0 \qquad \textrm{if and only if} \qquad k = 1.\label{unboundedidealdeflambda*} \ee
  For this choice of $\lambda$ we must have $\hat{\zeta}(k) = 0$ for $k \neq 1$, and thus evenness dictates that $\zeta \in \textrm{span}{(\cos{(\bar{x})})}$.  We conclude the null space is one-dimensional, proving (b).  
  
  On the other hand, suppose that $\xi \in \mathcal{R}(\mathcal{G}_\eta(\lambda^*, 0))$.  Then
  \[ \hat{\xi}(k) = m(k; \lambda^*) \hat{\zeta}(k), \qquad \textrm{for } k \geq 0,\]
  and, in particular, $\hat{\xi}(1) = 0$, by the definition of $\lambda^*$ in \eqref{unboundedidealdeflambda*}.  It follows that $\hat{\xi}(1) = 0$ is a necessary condition for inclusion in the range.  Conversely, if $\xi$ is any element of $Y$ with $\hat{\xi}(1) = 0$, then $\zeta \in X$ defined by
  \[ \hat{\zeta}(k) := \frac{1}{m(k; \lambda^*)} \hat{\xi}(k) \qquad k \neq 1,\]
  is in the preimage of $\xi$ under $\mathcal{G}_\eta(\lambda^*, 0)$.  Note that this definition is permissible since $m(k;\lambda^*)$ is nonvanishing for $k \neq 1$, again by \eqref{unboundedidealdeflambda*}.  We have therefore shown that $\xi \in \mathcal{R}(\mathcal{G}_\eta(\lambda^*, 0))$ if and only if $\hat{\xi}(1) = 0$, which is (iii).  Of course, this implies immediately that the codimension of the range is one, and so (i) follows.  \qquad
  \end{proof}
  \begin{remark}  Equation \eqref{unboundedidealdeflambda*} in fact gives an explicit definition for the point of bifurcation: 
\[ \lambda^* = \sqrt{ p_0^2 \coth{1} + g\jump{\rho}}.\]
\end{remark}

\pfthm{\ref{unboundedideallocalbifurcationtheorem}}  In the previous lemma, we confirmed hypothesis (iii) of Theorem \ref{crandallrabinowitz}; while (i) and (ii) clearly hold.  All that remains is the transversality condition, hypothesis (iv). But observe that 
  \[{(\langle \mathcal{G}_{\eta \lambda}(\lambda^*,0), \zeta^* \rangle)}^{\widehat{~}}\, (k) = -4\lambda^* k \widehat{\zeta^*}(k) = -4 \lambda^* \delta_{1k}.\]
By Lemma \ref{unboundedidealGmultiplier} (c), $\langle \mathcal{G}_{\eta \lambda}(\lambda^*,0), \zeta^* \rangle $ cannot be an element of the range of $\mathcal{G}_\eta( \lambda^*,0)$.  The statement of the theorem then follows from a straightforward application of Theorem \ref{crandallrabinowitz}.  \qquad \endproof

\section{Local bifurcation for shear flow in the atmosphere} \label{unboundedshearsection}

In fact, a fairly simple extension of Theorem
\ref{unboundedideallocalbifurcationtheorem} is possible when we
assume that the flow in the atmosphere region has constant
vorticity.  
Let  
\[ \gamma  = \left\{ \begin{array}{ll} \gamma_0 & \textrm{in } \Omega^{(1)} \\
0 & \textrm{in } \Omega^{(2)}\end{array} \right.\]
where $\gamma_0$ is a fixed constant.  Rather than study the relative stream function directly, we instead look at the perturbation of the stream function for the background shear flow.  In other words, let 
\[ \widetilde{\psi} := \psi + \frac{\gamma}{2} y^2.\] 
so that
\[ \Delta \widetilde{\psi} = \Delta \psi + \gamma = 0, \qquad
\textrm{in } \Omega.\] We shall call $\widetilde{\psi}$ the
\emph{modified stream function}.  The boundary conditions for
$\widetilde{\psi}$ follow naturally from those for $\psi$
enumerated in \eqref{unboundedidealpsieq}.  Notice, however, that
the interpretation of the condition $\nabla^\perp
\widetilde{\psi} \to (-\lambda, 0)$ as $y \to \infty$ is
different: we are now imposing the precise way in which the
\emph{shear} of the velocity field approaches infinity, instead
of the limiting value of the velocity. Even more importantly, we
point out that $\widetilde{\psi}$ is not continuous due to the
jump in $\gamma$ over the interface.  The appropriate choice of
Banach spaces in this setting is therefore
\begin{align*}
\widetilde{\mathscr{S}}_0 & := \{ (u,v,\varrho, \eta) \in \mathscr{S} : \exists U < c,  (u - U y,v) \to (0, 0) \textrm{ uniformly as } y \to \infty \} \\
\widetilde{\mathscr{S}}_0^\prime & := \{ (Q, \widetilde{\psi}, \eta) : (Q,\psi, \eta) \in \mathscr{S}^\prime, \, \nabla^\perp \widetilde{\psi} \to (-\lambda, 0) \textrm{ uniformly as } y \to \infty \}.  \end{align*}

With the notation established, we can now state our main result for the shear flow.
\begin{theorem} [Local bifurcation for unbounded shear
  atmosphere] \label{unboundedshearlocalbifurcationtheorem}
  Suppose the following local bifurcation condition holds \be
  p_0^2 \coth{1} + g\jump{\rho} -\frac{1}{4}(\gamma_0^-)^2
  +\frac{1}{2} \gamma_0 \gamma_0^- > 0,\label{unboundedshearlbc}
  \ee where $\gamma_0^{-} := \min\{ \gamma_0, 0\}$, then the
  following statements hold.
\begin{romannum}
\item There exists a continuous curve of non-laminar solutions to the stream function equation for irrotational flow in the air and water, and unbounded atmosphere region  \eqref{unboundedidealpsieq}
\[ \mathcal{C}_{\mathrm{loc}}^{\prime} = \{ (Q(s),\widetilde{\psi}(s), \eta(s)) \in \widetilde{\mathscr{S}}_0^\prime : |s| < \epsilon \}, \]
for $\epsilon > 0$ sufficiently small, such that $(Q(0), \psi(0)) = (Q^*, \widetilde{\Psi}_0^*)$, and, in a sufficiently small neighborhood of $(Q^*, \widetilde{\Psi}_0^*)$ in $\mathbb{R}\times X$, $\mathcal{C}_{\mathrm{loc}}^\prime$ comprises all non-laminar solutions. In particular, 
\[ \eta(s)(\cdot) = s \cos{(\cdot)} + o(s) \qquad \textrm{in } C_{\mathrm{per}}^{2+\alpha}(\mathbb{R}).\]
\item There is a corresponding continuous curve 
\[ \mathcal{C}_{\mathrm{loc}} = \{ (Q(s), u(s), v(s), \varrho(s), P(s), \eta(s)) \in \widetilde{\mathscr{S}}_0: |s| < \epsilon \}\]
of small amplitude solutions to the Eulerian problem which likewise captures all non-laminar solutions in a sufficiently small neighborhood of the point of bifurcation.  
\end{romannum}
\end{theorem}
\begin{remark}  As in Theorem \ref{unboundedideallocalbifurcationtheorem}, a simple generalization of this result is possible we have depth $d > 0$ and allow perturbations of the laminar flow with minimal period $2\pi k^*$, where $k^*$ is the smallest nonnegative integer satisfying 
\[ \frac{p_0^2}{d^2} k^* \coth{(k^*d)} + g\jump{\rho} - \frac{1}{4}(\gamma_0^-)^2 + \frac{1}{2} \gamma_0 \gamma_0^- > 0.\]
 We also note that when $\gamma_0 \geq 0$, \eqref{unboundedshearlbc} reduces to \eqref{unboundedideallbc}.
\end{remark}

\subsection{The transformed problem} \label{unboundedsheartranssection}

Define the transformations $T$ as in \S\ref{unboundedidealtranssection}, and put
\[ \widetilde{\Psi}  := \widetilde{\psi} \circ S.\]
Since $\widetilde{\psi}$ is harmonic in the unknown domain, $\widetilde{\Psi}$ will be in the kernel of the elliptic operator $\mathcal{E}(\eta)$, just as $\Psi$ was in the previous section.  The main difference will come in the boundary condition, and, most significantly, in the transmission boundary condition.  Writing $\widetilde{\Psi} - \gamma \bar{y}^2 /2 = \Psi$ and inserting this in to the Bernoulli equation, we find that the entire problem is equivalent to the vanishing of the operator
\be \begin{split}  \widetilde{\mathcal{G}}(\lambda, \eta) &:= \jump{ (C_{ij}(\eta)\partial_i \widetilde{\Psi}(\eta, \lambda))^2} + 2 \jump{ \gamma S_2(\eta) C_{i2}(\eta)  \partial_i \widetilde{\Psi}(\lambda, \eta) } \\
& \qquad  + 2 g\jump{\rho} (H(\eta))|_{\bar{y} = 0} - Q, \end{split} \label{unboundedsheardefwtG} \ee 
where $\widetilde{\Psi}(\lambda, \eta)$ is defined to be the (unique) solution to 
\be \left\{ \begin{array}{ll} 
\mathcal{E}(\eta) \widetilde{\Psi} = 0 &\textrm{in } \Omega_0 \setminus \{ \bar{y} = 0\} \\
\widetilde{\Psi}^{(1)} -\displaystyle \frac{\gamma_0}{2} S_2^2 = 0 & \textrm{on } \bar{y} = 0 \\
\widetilde{\Psi}^{(2)}  = 0 & \textrm{on } \bar{y} = 0 \\
\widetilde{\Psi} = - p_0 & \textrm{on } \bar{y} = -1 \\
\nabla^\perp \widetilde{\Psi} \to (-\lambda, 0) & \textrm{as } \bar{y} \to \infty. \end{array} \right. \label{unboundeshearwtPsieq} \ee

The laminar solution $\widetilde{\Psi}_0 = \widetilde{\Psi}_0(\lambda, 0)$ is the same as for the ideal flow, i.e. $\widetilde{\Psi}_0 = \Psi_0$.  This is simply because the only way in which the vorticity appears in \eqref{unboundeshearwtPsieq} is as a coefficient of $S_2$, but $S_2(0) = 0$.  We record this fact in the following lemma.

\begin{lemma}[Laminar flows] \label{unboundedshearlaminarlemma} There exists a one-parameter family of laminar solutions $(Q(\lambda), 0)$ to \eqref{unboundedGequation} for $\lambda \geq 0$.  The corresponding family of modified laminar transformed streamed function $\widetilde{\Psi}_0(\lambda)$ are given by
\be \widetilde{\Psi}_0(\lambda, \bar{y}) = \left\{ \begin{array}{ll} -\lambda \bar{y} & \textrm{for } \bar{y} > 0, \\
p_0 \bar{y} & \textrm{for } y \leq 0, \end{array} \right. \label{unboundedshearwtPsi0} \ee
and
\be Q(\lambda) := 2(\lambda^2 - p_0^2) + 2g\jump{\rho}.  \label{unboundedshearQ} \ee
\end{lemma}
\begin{remark}  Recalling the definition of the modified stream function, we have
\[ \widetilde{\Psi}_0 = \Psi_0 + \frac{\gamma}{2} y^2, \]
where $\Psi_0$ is the stream function for the laminar flow.  Thus,
\[ \Psi_0^\prime = \widetilde{\Psi}_0^\prime -\gamma_0 y = -\lambda - \gamma_0 y.\]
In other words, if $\gamma_0 < 0$,  then there is a critical layer in the air at $\bar{y} = |\gamma_0|/\lambda$.  Of course, because $\Psi_0^{\prime\prime\prime}$ vanishes identically, this will be a neutered layer in the sense discussed in the introduction.
\end{remark}

\subsection{Proof of local bifurcation} \label{unboundedshearprooflocalsection}

Next consider the linearization of the operators $\widetilde{\mathcal{E}}$ and $\widetilde{\mathcal{G}}$ around the laminar flows.  Let the variation of $\widetilde{\Psi}$ be denoted by $\widetilde{\Phi}$, and otherwise adopt the same notation as in the ideal atmosphere case considered in\S\ref{unboundedidealsection}.  Then
\begin{align}
 \langle \widetilde{\mathcal{G}}_\eta(\lambda,0), \zeta \rangle & = 2 \jump{(-\widetilde{\Psi}_0 \partial_2 h + \partial_2 \widetilde{\Phi}) \widetilde{\Psi}_0^\prime} - 2 \gamma_0 \zeta \lambda + 2 g \jump{\rho} \zeta \nonumber \\
 & = 2 \left( g \jump{\rho} - \jump{(\widetilde{\Psi}_0^\prime)^2} -2 \gamma_0 \lambda \right) \zeta +2\jump{ \widetilde{\Psi}_0^\prime \partial_2 \widetilde{\Phi}}. \label{unboundedwtGetaeq} \end{align} 
On the other hand, $\widetilde{\Psi}$ solves \eqref{unboundedidealdefPhi1}, and so we may define $\widetilde{\Upsilon} := \widetilde{\Psi}_0^\prime \widetilde{\Phi} - (\widetilde{\Psi}_0^\prime)^2 h$, and from an identical argument we find
\[ \jump{\partial_2 \widetilde{\Upsilon}} = -\jump{(\widetilde{\Psi}_0^\prime)^2} \zeta +  \jump{\widetilde{\Psi}_0^\prime \partial_2 \widetilde{\Psi} } ,\]
which can be understood as a Fourier multiplier using \eqref{unboundedideald2Upsilonsymbol}:
\be \langle \widetilde{\mathcal{G}}_\eta(\lambda, 0) , \zeta \rangle = 2(p_0^2 D \coth{D} - \lambda^2 D + g\jump{\rho} - \gamma_0 \lambda) \zeta. \label{unboundedshearwtGmultiplier} \ee
Observe that the only difference between the symbol in \eqref{unboundedshearwtGmultiplier} and that in \eqref{unboundedidealGmultiplier} is the $-2\gamma_0 \lambda \zeta$ term.

\begin{lemma}[Null space and range]  \label{unboundedshearnullspacelemma} Assume that the local bifurcation condition holds. Then there exists a $\lambda^* > 0$ such that the following statements hold:
  \begin{romannum}
  \item  $\widetilde{\mathcal{G}}_\eta(\lambda^*, 0)$ is a Fredholm operator of index 0; 
  \item   $\mathcal{N}(\widetilde{\mathcal{G}}_\eta(\lambda^*, 0))$ is one-dimensional and spanned by $\zeta^*(\bar{x}) :=  \cos{\bar{x}}$; moreover
  \item   $\xi \in \mathcal{R}(\widetilde{\mathcal{G}}_\eta(\lambda^*, 0))$ if and only if $\hat{\xi}(1) = 0$.  \end{romannum}\end{lemma}
\begin{proof}  From \eqref{unboundedshearwtGmultiplier} it is clear that $\zeta \in \mathcal{N}(\widetilde{\mathcal{G}}_\eta(\lambda^*, 0))$ if and only if 
\[ \widetilde{m}(k; \lambda^*) \hat{\zeta}(k) = 0, \qquad \textrm{for all } k \geq 0, \]
where
\[ \widetilde{m}(k; \lambda) := p_0^2 k \coth{k} - \lambda^2 k + g \jump{\rho} - \gamma_0 \lambda. \]
First we note that, for any fixed $\lambda \geq 0$, $k \mapsto \widetilde{m}(k; \lambda)$ is a strictly decreasing function and thus any root is unique.  Since we are primarily interested in the case where the perturbations of the laminar flow are minimally $2\pi$-periodic, we wish to consider $\lambda$ for which  $m(1; \lambda) = 0$.  Differentiating, it becomes clear that $m(1;\lambda)$ is a concave function of $\lambda$ tending to $-\infty$ as $\lambda \to \infty$ and that the maximum occurs at
 \[ \lambda = (-\frac{\gamma_0}{2})^+ = -\frac{1}{2} \gamma_0^-.\]
  The local bifurcation condition \eqref{unboundedshearlbc} implies that 
  \[ \max_{\lambda \geq 0} \widetilde{m}(1;\lambda) \geq  \widetilde{m}(1; -\frac{1}{2} \gamma_0^-) > 0.\]  
  By continuity, we have that there exists a $\lambda^*$ such that $\widetilde{m}(1; \lambda^*) = 0$ and $\widetilde{m}(k; \lambda^*) \neq 0$ for $k \neq 1$.  This proves (ii).  The proof of the remaining parts follows exactly as in Lemma \ref{unboundedidealnullspacelemma}.  \qquad\end{proof}

\pfthm{\ref{unboundedshearlocalbifurcationtheorem}}  We have already laid the groundwork for a bifurcation argument via Crandall--Rabinowitz.  The only detail that remains to be checked is the transversality condition.  With that in mind, we calculate 
\begin{align*} (\langle \widetilde{\mathcal{G}}_{\eta \lambda}(\lambda^*, 0), \zeta^* \rangle)^{\widehat{~}}(k) &= -4\lambda^* k \widehat{\zeta^*}(k) -2\gamma_0 \widehat{\zeta^*}(k)\\
& = -4\lambda^* \delta_{1k} -2\gamma_0 \delta_{1k}. \end{align*}
In light of Lemma \ref{unboundedshearnullspacelemma} (c), the only way for $\zeta^*$ to be an element of $\mathcal{R}(\widetilde{\mathcal{G}}_\eta(0, \lambda^*))$ is if $2\lambda^* = -\gamma_0$.  But this scenario is excluded by the local bifurcation condition \eqref{unboundedshearlbc}, as $\widetilde{m}(1; -\gamma_0^-/2) > 0$ while $\widetilde{m}(1; \lambda^*) = 0$.  This confirms that the transversality condition holds, and thus we obtain the theorem via a routine application of Theorem \ref{crandallrabinowitz}. \qquad 
\endproof

\Appendix \section*{}  Here we present two standard theorems.  The first is the classical work of Crandall--Rabinowitz on bifurcation from simple (generalized) eigenvalues: 

\begin{theorem} \emph{(Crandall--Rabinowitz \cite{crandall1971bifurcation})} Let $X$ and $Y$ be Banach spaces, $I \subset \mathbb{R}$ an open interval with $\lambda_* \in I$.  Suppose that $\mathcal{F} : I \times X \to Y$ is a continuous map with the following properties:
\begin{romannum}
\item $\mathcal{F}(\lambda, 0) = 0$, for all $\lambda \in I$;
\item $D_1 \mathcal{F}$, $D_2 \mathcal{F}$ and $D_1 D_2 \mathcal{F}$ exist and are continuous, where $D_i$ denotes the Fr\'echet derivative with respect to the $i$-th coordinate;
\item  $D_2 \mathcal{F}(\lambda_*,0)$ is a Fredholm operator of index $0$, in particular, the null space is one-dimensional and spanned by some element $w_*$;
\item $D_1 D_2 \mathcal{F}(\lambda_*, 0)w_* \notin \mathcal{R}(D_2 \mathcal{F}(\lambda_*,0))$.
\end{romannum}
 Then there exists a continuous local bifurcation curve $\{(\lambda(s), w(s)) \in \mathbb{R} \times X : |s| < \epsilon \}$ with $\epsilon > 0$ sufficiently small such that $(\lambda(0), w(0)) = (\lambda_*, w_*)$, and
 \[ \{ (\lambda, w) \in \mathcal{U} : w \neq 0, \mathcal{F}(\lambda, w) = 0 \} = \{(\lambda(s), w(s)) \in \mathbb{R} \times X : |s| < \epsilon \}\]
 for some neighborhood $\mathcal{U}$ of $(\lambda_*,0)$ in $\mathbb{R} \times X$.  Moreover, we have 
 \[ w(s) = sw_* +o(s), \qquad \textrm{in } X,~|s| < \epsilon. \]
 If $D_2^2 \mathcal{F}$ exists and is continuous, then the curve is of class $C^1$.  
\label{crandallrabinowitz}
\end{theorem}
The reader should view this as a special case of the more general Lyapunov--Schmidt reduction procedure, a good discussion of which can be found in \cite{chow1982bifurcation}.  

The second theorem is on the Fredholm solvability and regularity of solutions to linear elliptic equations with jump conditions across an interface.  A classical reference for this is \cite{ladyzhenskaya1968linear}.  We quote here a parsed version of Theorem 16.1 from that book, incorporating the discussion preceding and following the theorem statement and simplifying the hypotheses to better match our needs.  

Let $\Omega = \Omega^{(1)} \cup \Omega^{(2)}$ be a domain in $\mathbb{R}^2$ with smooth boundary of class $C^{2+\alpha}$, where $\Omega^{(1)}$ and $\Omega^{(2)}$ are the connected components of $\Omega$ and their shared boundary $\mathcal{I} := \partial \Omega^{(1)} \cap \partial\Omega^{(2)}$ is a simple curve (which will also be of class $C^{2+\alpha}$).  Consider the elliptic problem 
\be \label{appendix:ellipticproblem} \left\{ \begin{array}{ll} \partial_{x_i} ( a_{ij} \partial_{x_j} u) + b_i \partial_{x_i} u + c u = f & \textrm{in $\Omega$} \\ 
\jump{u} = 0 & \textrm{on $\mathcal{I}$} \\
\jump{d \partial_\nu u} + \sigma u = g  & \textrm{on $\mathcal{I}$} \\
u = h & \textrm{on $\partial\Omega \setminus \mathcal{I}$}.  \end{array} \right.\ee
Here we are using summation conventions, $\partial_\nu$ denotes the conormal derivative with respect to $a_{ij}$, $\sigma$ is a real constant,  and the coefficients and forcing terms are assumed to have the following regularity:  
\be \label{appendix:regularityassumptions} 
\begin{array}{l} a_{ij} \textrm{ uniformly elliptic},~d \geq d_0 > 0, \\
b_i,~c,~f \in C^{\alpha}(\overline{\Omega^{(1)}}) \cap C^{\alpha}(\overline{\Omega^{(2)}}), \\
a_{ij}, ~d, ~g\in C^{1+\alpha}(\overline{\Omega^{(1)}}) \cap C^{1+\alpha} (\overline{\Omega^{(2)}}), \\
 h \in C^{2+\alpha}(\overline{\Omega^{(1)}}) \cap C^{2+\alpha} (\overline{\Omega^{(2)}}). \end{array}
\ee

\begin{theorem} \emph{(Ladyzhenskaya--Ural'tseva \cite[Theorem 16.1]{ladyzhenskaya1968linear})}
\label{ellipticregtheorem}  Consider the elliptic problem \eqref{appendix:ellipticproblem} with the assumptions listed in \eqref{appendix:regularityassumptions}.  Then the following statements hold.
\begin{romannum}
\item \emph{(Fredholm solvability)} The existence of a classical solution
\[ u \in C^{2+\alpha}(\overline{\Omega} \setminus \mathcal{I}) \cap C^{1+\alpha}(\overline{\Omega^{(1)}}) \cap C^{1+\alpha}(\overline{\Omega^{(2)}})\]
is implied by the uniqueness of solutions.     
\item \emph{(Elliptic regularity)} Note that, due to the fact that the $\jump{u} = 0$, a classical solution can be extended to $\overline{\Omega}$ as a continuous function.  In fact, if a classical solution $u$ exists, then its extension to $\overline{\Omega}$ is of class  $C^\alpha{(\overline{\Omega})}$.  \end{romannum}
\end{theorem}

\bibliographystyle{siam}
\bibliography{steadywaves}

\end{document}